\documentclass{siamart190516}

\usepackage{lipsum}
\usepackage{amsfonts}
\usepackage{graphicx}
\usepackage{epstopdf}
\usepackage{algorithmic}
\ifpdf
  \DeclareGraphicsExtensions{.eps,.pdf,.png,.jpg}
\else
  \DeclareGraphicsExtensions{.eps}
\fi

\newsiamremark{hypothesis}{Hypothesis}
\crefname{hypothesis}{Hypothesis}{Hypotheses}
\newsiamthm{claim}{Claim}

\headers{Least-squares neural network method}{Z. Cai, J. Chen AND M. Liu}

\title{Least-Squares ReLU Neural Network (LSNN) Method\\[1mm] for
Linear Advection-Reaction Equation\thanks{Submitted, January 6, 2021. Revised April 15 and April 28, 2021.}}

\author{Zhiqiang Cai\thanks{Department of Mathematics, Purdue University, 150 N. University Street, West Lafayette, IN 47907-2067 
  (\email{caiz@purdue.edu}, \email{chen2042@purdue.edu}).}
\and Jingshuang Chen\footnotemark[2]
\and Min Liu\thanks{School of Mechanical Engineering, Purdue University, 585 Purdue Mall,
West Lafayette, IN 47907-2088(\email{liu66@purdue.edu}). }}

\usepackage{amsopn}

\ifpdf
\hypersetup{
  pdftitle={FOSLS},
  pdfauthor={Z.Cai, J.Chen, M.Liu}
}
\fi

\usepackage{mathtools}
\usepackage[utf8]{inputenc}
\usepackage{bm}
\usepackage{indentfirst}
\usepackage{subfigure}
\usepackage[export]{adjustbox}
\usepackage{scalerel}

\newcommand{\R}{\mathbb{R}}

\usepackage{graphicx}
\usepackage{diagbox}
\usepackage{pifont}
\newcommand{\vertiii}[1]{{\left\vert\kern-0.25ex\left\vert\kern-0.25ex\left\vert #1 
    \right\vert\kern-0.25ex\right\vert\kern-0.25ex\right\vert}}

\newcommand{\btheta}{\mbox{\boldmath${\theta}$}}

\newcommand{\bomega}{\mbox{\boldmath$\omega$}}
\newcommand{\bxi}{\mbox{\boldmath$\xi$}}

\newcommand{\jump}[1]{[\![ #1]\!]}

\def\bb{{\bf b}}
\def\bc{{\bf c}}
\def\bv{{\bf v}}
\def\bx{{\bf x}}
\def\by{{\bf y}}

\def\cC{{\cal C}}
\def\cL{{\cal L}}
\def\cM{{\cal M}}
\def\cN{{\cal N}}
\def\cP{{\cal P}}
\def\cS{{\cal S}}
\def\cT{{\cal T}}

\begin{document}

\maketitle

\begin{abstract}
This paper studies least-squares ReLU neural network method for solving the linear advection-reaction problem with discontinuous solution. The method is a discretization of an equivalent least-squares formulation in the set of neural network functions with the ReLU activation function.
The method is capable of approximating the discontinuous interface of the underlying problem automatically through the {\it free} hyper-planes of the ReLU neural network and, hence, outperforms mesh-based numerical methods in terms of the number of degrees of freedom.
Numerical results of some benchmark test problems show that the method can not only approximate the solution with the least number of parameters, but also avoid the common Gibbs phenomena along the discontinuous interface. Moreover, a three-layer ReLU neural network is necessary and sufficient in order to well approximate a discontinuous solution with an interface in $\R^2$ that is not a straight line.
\end{abstract}



\section{Introduction}

During the past several decades, numerical methods for linear advection-reaction equations have been intensively studied by many researchers and many numerical schemes have been developed. When inflow boundary data is discontinuous, so is the solution. It is well-known that 
traditional mesh-based numerical methods often exhibit oscillations near a discontinuity (called the Gibbs phenomena).
Such spurious oscillations are unacceptable for many applications (see, e.g, \cite{hesthaven2007nodal}). 
To eliminate or reduce the Gibbs phenomena, finite difference and finite volume methods often use numerical techniques such as limiters, filters, ENO/WENO, etc. \cite{gottlieb1997gibbs,hesthaven2017numerical,hesthaven2007nodal, leveque1992numerical}; and finite element methods 
usually employ discontinuous finite elements \cite{brezzi2004discontinuous,dahmen2012adaptive, demkowicz2010class} and/or adaptive mesh refinement (AMR) to generate locally refined elements along discontinuous interfaces (see, e.g., \cite{burman2009posteriori, houston1999posteriori, houston2000posteriori}). 

Recently, there has been increasing interests in using deep neural networks (DNNs) to solve partial differential equations (see, e.g., \cite{cai2020deep,raissi2019physics,Sirignano18}). DNNs produce a large class of functions through compositions of linear transformations and activation functions. One of the striking features of DNNs is that this class of functions is not subject to a hand-crafted geometric mesh or point cloud as are the traditional, well-studied finite difference, finite volume, and finite element methods. The physical partition of the domain $\Omega$, formed by free hyper-planes, can automatically adapt to the target function. This is much better than the AMR generated mesh because AMR is based on a geometric mesh and subject to mesh conformity; moreover, it is not easy to remove unnecessary elements or points. This paper will make use of this powerful approximation property of DNNs for solving linear advection-reaction problem with discontinuous solution. 

DNN functions are nonlinear functions of the parameters. Hence, the advection-reaction equation will be discretized through least-squares principles. In the context of finite element approximations, several least-squares methods have been studied (see, e.g., \cite{bochev2001improved, bochev2016least,bochev2001comparative,carey1988least,de2004least,de2005numerical,liu2020adaptive}). Basically, there are two least-squares formulations which are equivalent to the original differential equation. One is a direct application of least-squares principle
(see, e.g., \cite{bochev2001improved, de2004least}) with a weighted $L^2$ norm for the inflow boundary condition, where the weight is the magnitude of the normal component of the advection velocity field.
The other is to apply the least-squares principle to an equivalent system of the underlying problem by introducing an additional flux variable (see \cite{de2005numerical, liu2020adaptive}). Some numerical techniques such as feedback least-squares finite element method \cite{bochev2016least}, adaptive local mesh refinement with proper finite elements \cite{liu2020adaptive}, etc. were introduced in order to reduce the Gibbs phenomena for problems with discontinuous solutions.

The purpose of this paper is to study the least-squares neural network (LSNN) method for solving the linear advection-reaction problem with discontinuous solution. The LSNN method is based on the least-squares formulation studied in (\cite{bochev2001improved, de2004least}), i.e., a direct application of the lease-squares principle to the underlying problem, and on the ReLU neural network as the class of approximating functions. 
The class of neural network functions enables the LSNN method to automatically approximate the discontinuous solution without using {\it a priori} knowledge of the location of the discontinuities. Compared to various AMR methods that locate the discontinuous interface through local mesh refinement, the LSNN method is much more effective in terms of the number of the degrees of freedom (see, e.g., Fig. \ref{test1_1_compare_figure}(c) and \ref{test1_1_figure}(c)). 

Theoretically, it is proved in \cite{de2004least} that the homogeneous least-squares functional is equivalent to a natural norm in the solution space $V_{\bm\beta}$ consisting of all square-integrable functions whose directional derivative along ${\bm \beta}$ is also square-integrable (see section 2). 
This equivalence enables us to prove Ce\'{a}'s lemma for the LSNN approximation, i.e., the error of the LSNN approximation is bounded by the approximation error of the set of ReLU neural network functions. This result is extended to the LSNN method with numerical integration as well. 
Even though approximation theory of the ReLU neural network has been intensively studied by many researchers (see, e.g., \cite{pinkus1999approximation} for work before 2000 and \cite{shen2019deep,siegel2020approximation}), we are not able to find a result which is applicable to the discontinuous solution of the advection-reaction problem. 

To explore how well the ReLU neural network approximates the discontinuous solution, we consider two-dimensional transport problem, i.e., (\ref{pde-1}) with $\hat{\gamma}=0$. When the boundary data $g$ is discontinuous at point $\bx_0\in\Gamma_-$, 
the solution of the transport problem is discontinuous across an interface: the streamline of the advection velocity field starting at $\bx_0$. The solution of this problem can be decomposed as the sum of a piece-wise constant function and a continuous piece-wise smooth function (see, e.g., (\ref{decop})). We show that the piece-wise constant function can be approximated well without the Gibbs phenomena by either a two- or a three-layer ReLU neural network with the minimal number of neurons depending on the shape of the interface (see Lemmas 3.1 and 5.1). Together with the universal approximation property, this implies that a two- or three-layer ReLU neural network is sufficient to well approximate
the solution of the linear transport problem without oscillation. These theoretical results are confirmed by numerical results. 

The procedure for determining the values of the parameters of the network is now a problem in nonlinear optimization even though the underlying PDE is linear. This high dimensional, nonlinear optimization problem usually has many solutions. In order to obtain the desired one, we need to start from a close enough first approximation, and a common way to do so is by the method of continuation. In this paper, we propose the method of model continuation through approximating the advection velocity field by a family of piece-wise constant vector fields. Numerical results for a test problem with variable velocity field show that this method is able to reduce the total number of the parameters significantly.

The paper is organized as follows. Section 2 introduces the advection-reaction problem, its least-squares formulation, and preliminaries. The ReLU neural network and the least-squares neural network are described and analyzed in section 3. Initialization for the two-layer neural network and the method of model continuation for initialization are presented in sections 4 and 6, respectively. Finally, numerical results for various benchmark test problems are given in section 5.

Standard notations and definitions are used for the Sobolev space $H^s(\Omega)^d$ and $H^s(\Gamma_{-})^d$ when $s \geq 0$. The associated norms with these two spaces are denoted by $\|\cdot\|_{s,\Omega}$ and $\|\cdot\|_{s,\Gamma_{-}}$, and their respective inner products are denoted as $(\cdot, \cdot)_{s,\Omega} $ and $(\cdot, \cdot)_{s,\Gamma_{-}} $. For $s=0$ case, $H^s(\Omega)^d$ is the same as $L^2(\Omega)^d$, then the norm and inner product are simply denoted as $\|\cdot\|$ and $(\cdot, \cdot)$, respectively. The subscripts $\Omega$ in the designation of norms will be suppressed when there is no ambiguity.

\section{Problem Formulation}
Let $\Omega$ be a bounded domain in ${\R}^d$ with
Lipschitz boundary, and denote the advective velocity field by $\bm{\beta}(\bx) = (\beta_1, \cdots, \beta_d)^T\in C^1(\bar{\Omega})^d$. Define the inflow and outflow parts of the boundary $\Gamma=\partial \Omega$ by
\begin{equation}
    \Gamma_- = \{\bx\in\Gamma :\, \bm{\beta}(\bx) \cdot \bm{n}(\bx) <0\}
    \quad\mbox{and}\quad
    \Gamma_+ = \{\bx\in\Gamma :\, \bm{\beta}(\bx) \cdot \bm{n}(\bx) > 0\},
\end{equation}
respectively, where $\bm{n}(\bx)$ is the unit outward normal vector to $\Gamma$ at $\bx\in \Gamma$.

As a model hyperbolic boundary value problem, we consider the linear advection-reaction equation 
\begin{equation}\label{pde}
    \left\{\begin{array}{rccl}
    \nabla \!\cdot\! (\bm \beta u) + \gamma u &= & f &\text{ in }\,\, \Omega, \\
    u&=&g &\text{ on }\,\, \Gamma_{-},
    \end{array}\right.
\end{equation}
where $\gamma \in C(\bar{\Omega})$, $f \in L^2(\Omega)$, and $g \in L^2(\Gamma_-)$ are given scalar-valued functions. We assume that there exist a positive constant $\gamma_0$ such that
\begin{equation}\label{gamma}
    \gamma (\bx) + \frac{1}{2}\nabla \cdot \bm{\beta} (\bx) \geq \gamma_0 >0\quad  \text{ for all } \bx\in \Omega.
\end{equation}
For simplicity of presentation, we also assume that $g$ is bounded so that streamline functions from $\Gamma_{-}$ to $\Gamma_{+}$ is not needed (see \cite{de2004least}).

Denote by $v_{\bm\beta} = \bm{\beta}\cdot \nabla v$ the directional derivative along the advective velocity field $\bm{\beta}$, then (\ref{pde}) may be rewritten as follows
\begin{equation}\label{pde-1}
    \left\{\begin{array}{rccl}
    u_{\bm\beta} + \hat{\gamma}\, u &=&f &\text{ in }\, \Omega, \\[2mm]
    u&=&g &\text{ on }\,\, \Gamma_{-},
    \end{array}\right.
\end{equation}
where $\hat{\gamma} = \gamma + \nabla\! \cdot\! \bm{\beta}$.
The solution space of (\ref{pde}) is given by
 \[
 V_{\bm\beta} = \{v\in L^2(\Omega): v_{\bm{\beta}}\in L^2(\Omega)\},
 \]
which is equipped with the norm as
 \[
 \vertiii{v}_{\bm\beta}= \left(\|v\|_{0,\Omega}^2 + \|v_{\bm\beta}
 \|_{0,\Omega}^2 \right)^{1/2}.
 \]
Denote the weighted $L^2(\Gamma_{-})$ norm over the inflow boundary by
 \[
 \|v\|_{-\bm{\beta}} 
 =\left<v,v\right>^{1/2}_{-\bm{\beta}} 
 =\left( \int_{\Gamma_-} |\bm{\beta}\! \cdot \!\bm{n}|\, v^2\,ds\right)^{1/2}.
 \]
The following trace and Poincar\'{e} inequalities are proved in \cite{de2004least} (see also \cite{bochev2016least}) that there exist positive constants $C_t$ and $C_p$ such that 
\begin{equation}\label{trace}
 \|v\|_{-\bm{\beta}} \leq C_t\, \vertiii{v}_{\bm\beta},\quad \forall\,\,v\in V_{\bm\beta}
\end{equation}
and 
\begin{equation}\label{trace1}
 \|v\|_{0,\Omega} \leq C_p\left(\|v\|_{-\bm{\beta}}  + D\, \|v_{\bm\beta}\|_{0,\Omega} \right), \quad \forall\,v\in V_{\bm\beta},
\end{equation}
respectively, where $D=\text{diam} (\Omega)$ is the diameter of the domain $\Omega$.

\begin{remark} 
Let $\cC$ be the streamline of the advection velocity field ${\bm \beta}$ starting at $\bx_0\in\Gamma_-$ in two dimensions.  Assume that the inflow boundary condition $g$ is discontinuous at $\bx_0$.
Then it is easy to see that the solution of {\em (\ref{pde})} is also discontinuous across $\cC$ because the restriction of the solution on $\cC$ satisfies the same differential equation but different initial condition.
Moreover, if $\hat{\gamma}=0$, then the jump of the solution along $\cC$ is a constant $|g(\bx_0^+)-g(\bx_0^-)|$,
where $g(\bx_0^+)$ and $g(\bx_0^-)$ are the values of $g$ at $\bx_0$ from different sides.
The streamline $\cC$ is referred to be the discontinuous interface.
\end{remark}

In the remainder of this section, we describe the least-squares (LS) formulation following \cite{bochev2016least, de2004least}. To this end, define the LS functional
 \begin{equation}\label{ls}
    \mathcal{L}(v;{\bf f}) = \|v_{\bm\beta} +\hat{\gamma}\, v-f\|_{0,\Omega}^2 +  \|v-g\|_{-\bm\beta}^2 
\end{equation}
for all $v\in V_{\bm\beta}$, where ${\bf f} = (f,g)$. Now, the corresponding least-squares formulation is to seek $u\in V_{\bm\beta}$ such that
\begin{equation}\label{minimization1}
    \mathcal{L}(u;{\bf f}) = \min_{\small v\in V_{\bm\beta}} \mathcal{L}(v;{\bf f}).
\end{equation}
It follows from the trace, triangle, and Poincar\'{e} inequalities and assumptions on ${\bm \beta}$ and $\gamma$ that the homogeneous LS functional $\mathcal{L}(v;{\bf 0})$ is equivalent to the norm $\vertiii{v}_{\bm\beta}^2$, i.e., there exist positive constants $\alpha$ and $M$ such that 
\begin{equation}\label{equiv}
\alpha\, \vertiii{v}_{\bm\beta}^2 
\leq \mathcal{L}(v;{\bf 0}) \leq M\, \vertiii{v}_{\bm\beta}^2.
\end{equation}
Furthermore, problem (\ref{minimization1}) has a unique solution $u\in V_{\bm\beta}$ satisfying the following {\it a priori} estimate
\begin{equation}\label{stability}
\vertiii{u}_{\bm\beta} \leq C\, \left(\|f\|_{0,\Omega} + \|g\|_{-\bm\beta}\right).
\end{equation}

Denote the bilinear and linear forms by
\[
a(u,v)=(u_{\bm\beta} + \hat{\gamma}\, u,
v_{\bm\beta} + \hat{\gamma}\, v) + \left<u,v\right>_{-\bm{\beta}}
\quad\mbox{and}\quad
f(v)=(f,
v_{\bm\beta} + \hat{\gamma}\, v) + \left<g,v\right>_{-\bm{\beta}},
\]
respectively. Then the minimization problem in (\ref{minimization1}) is to find $u\in V_{\bm\beta}$ such that
\begin{equation}\label{vp}
    a(u,v)=f(v), \quad\forall\,\, v\in V_{\bm\beta}.
\end{equation}

\section{Least-Squares Neural Network Method}

This section describes deep neural networks and the corresponding least-squares method for linear transport equations. 

We consider a deep neural network (DNN) with a scalar-valued output as
\[
\cN:\, \bx\in\R^{d}
\longrightarrow \cN(\bx)\in\R.
\]
The DNN function $\cN(\bx)$ is typically represented as compositions of many layers of functions:  
\begin{equation}\label{DNN}
 \cN(\bx)=N^{(L)} \circ \cdots N^{(2)}\circ N^{(1)}(\bx),
\end{equation}
where the symbol $\circ$ denotes the composition of functions, and $L$ is the depth of the network. In this case, $N^{(l)}$ is called the $l^{th}$ layer of the network. All layers except the last one $N^{(L)}$ are called hidden layers. A layer 
$N^{(l)}: \R^{n_{l-1}} \rightarrow \R^{n_{l}}$ is defined as a composition of a linear transformation $T^{(l)}: \R^{n_{l-1}} \rightarrow \R^{n_{l}}$ and an activation function $\sigma: \R \rightarrow \R$ as follows:
\begin{equation}\label{layerdef}
   N^{(l)}(\bx^{(l-1)})=\sigma \big( T^{(l)}(\bx^{(l-1)})\big)= \sigma  (\bomega^{(l)}\bx^{(l-1)}-\bb^{(l)})
   \quad\mbox{for } \bx^{(l-1)}\in \R^{n_{l-1}},
\end{equation}
where $\bomega^{(l)} 
\in \R^{n_{l}\times n_{l-1}}$, $\bb^{(l)}\in \R^{n_{l}}$, $\bx^{(0)}=\bx$, and application of $\sigma$ to a vector is defined component-wise. There is typically no activation function in the output layer.
Components of $\bomega^{(l)}$ and $\bb^{(l)}$ are called weights and bias, respectively, and are parameters to be determined (trained). 
Each component of the vector-valued function $N^{(l)}$ is interpreted as a neuron and the dimension $n_{l}$ defines the width or the number of neurons of the $l^{\text{th}}$ layer in a network. This paper will use the popular rectified linear unit (ReLU) activation function defined by

\begin{equation}\label{tau-k}
 \sigma(t) = \max\{0,\,t\}
 =\left\{\begin{array}{rclll}
 0, & \mbox{if }  t\leq 0,\\[2mm]
 t, & \mbox{if } t >0.
 \end{array}\right.
 \end{equation}

For given integers $\{n_l\}_{l=1}^L$, denote the set of DNN functions by
\[
\cM({\small\btheta},L)=\big\{\cN(\bx)= N^{(L)} \circ \cdots \circ N^{(1)}(\bx) :\,  \bomega^{(l)} 
\in \R^{n_{l}\times n_{l-1}},\,\, \bb^{(l)}\in \R^{n_{l}} \mbox{ for } l=1,...,L
\big\},
\]
where $N^{(l)}(\bx^{(l-1)})$ is defined in (\ref{layerdef}) and ${\small\btheta}$ denotes all parameters: $\bomega^{(l)}$ and $\bb^{(l)}$ for $l=1,...,L$.
It is easy to see that $\cM({\small\btheta},L)$ is a subset of $V_{\bm\beta}$, but not a linear subspace. The least-squares approximation is to find $u_{_N}(\bx;{\small\btheta}^*) \in \cM({\small\btheta},L)$ such that 
\begin{equation}\label{L-NN}
     \mathcal{L}\big(u_{_N}(\bx;{\small\btheta}^*);\,{\bf f}\big)
     = \min\limits_{v\in \cM({\scriptsize\btheta},L)} \mathcal{L}\big(v(\bx;{\small\btheta});\,{\bf f}\big)
     = \min_{{\scriptsize \btheta}\in\R^{N}}\mathcal{L}\big(v(\bx;{\small\btheta});\,{\bf f}\big),
\end{equation}
where $N$ is the total number of parameters in $\cM({\small\btheta},L)$ given by
\[
N=M_d(L) =\sum^L_{l=1} n_{l}\times (n_{l-1}+1).
\]

\begin{lemma}
Let $u$ and $u_{_N}$ be the solutions of problems {\em (\ref{ls})} and {\em (\ref{L-NN})}, respectively. Then we have
\begin{equation}\label{Cea-L}
    \vertiii{u-u_{_N}}_{\bm\beta}
    \le \left(\dfrac{M}{\alpha}\right)^{1/2} \inf_{v\in \cM({\scriptsize\btheta},L)} \vertiii{u-v}_{\bm\beta},
\end{equation}
where $\alpha$ and $M$ are constants in {\em (\ref{equiv})}.
\end{lemma}

\begin{proof}
For any $v\in \cM({\small\btheta},L)\subset V_{\bm\beta}$, it follows from the coercivity and continuity of the homogeneous functional $ \mathcal{L}\big(v;\,{\bf 0}\big)$ in (\ref{equiv}), problem (\ref{pde}), and (\ref{L-NN}) that
\begin{eqnarray*}
\alpha\, \vertiii{u-u_{_N}}^2_{\bm\beta}
&\le &  \mathcal{L}\big(u-u_{_N};\,{\bf 0}\big)
= \mathcal{L}\big(u_{_N}(\bx;{\small\btheta}^*);\,{\bf f}\big)\\[2mm]
&\le & \mathcal{L}\big(v(\bx;{\small\btheta});\,{\bf f}\big)
= \mathcal{L}\big(u-v;\,{\bf 0}\big)
\le M \,\vertiii{u-v}^2_{\bm\beta},
\end{eqnarray*}
which implies (\ref{Cea-L}). This completes the proof of the lemma.
\end{proof}

For a given vector $\bxi\in \R^d$ and $c\in \R$, assume that the hyper-plane $\cP:\,\bxi\cdot\bx =c$ divides the domain $\Omega$ into two non-empty subdomains $\Omega_1$ and $\Omega_2$, i.e., 
\[
\Omega_1=\{\bx\in\Omega :\, \bxi\cdot\bx <c\}
\quad\mbox{and}\quad
\Omega_2=\{\bx\in\Omega :\, \bxi\cdot\bx >c\}.
\]
Let $\chi(\bx;\bxi,c)$ be a piece-wise constant function defined by 
 \[
 \chi(\bx;\bxi,c)=\left\{\begin{array}{ll}
 \alpha_1, & \bx \in \Omega_1,\\[2mm]
 \alpha_2, & \bx \in \Omega_2.
 \end{array}
 \right.
 \]
 
\begin{lemma}
Let $p(\bx)$ be a two-layer neural network function given by
\[
p(\bx)=\alpha_1 +
\dfrac{\alpha_2-\alpha_1}{2\varepsilon} \Big(\sigma(\bxi\cdot\bx -c + \varepsilon) - \sigma(\bxi\cdot\bx -c-\varepsilon)\Big)
\]
for any $\varepsilon>0$ such that intersections between the domain $\Omega$ and the hyper-planes $\bxi\cdot\bx =c\pm\varepsilon$ are not empty. Then we have 
 \begin{equation}\label{chi-est}
     \|\chi - p\|_{0,\Omega}
     =\left(\|\chi - p\|^2_{0,\Omega} + \|\chi_{\scriptsize\bm\eta} - p_{\scriptsize\bm\eta}\|^2_{0,\Omega}\right)^{1/2} 
     \leq {\dfrac{1}{\sqrt{6}}}\, D^{(d-1)/2}\, \big|\alpha_1-\alpha_2\big|\, \sqrt{\varepsilon},
 \end{equation}
where ${\bm\eta}$ is a vector normal to ${\bxi}$ and $D$ is the diameter of the domain $\Omega$.
\end{lemma}

\begin{proof}
Let 
\[
\Omega_\varepsilon=\Omega_{\varepsilon,1}\cup \Omega_{\varepsilon,2}
\equiv \{\bx\in \Omega :\, c-\varepsilon< \bxi\cdot\bx < c\}\cup
\{\bx\in\Omega :\, c< \bxi\cdot\bx < c+\varepsilon\}.
\]
The equality in (\ref{chi-est}) follows from the fact that $\chi_{\scriptsize\bm\eta} - p_{\scriptsize\bm\eta}=0$. To show the validity of the inequality in (\ref{chi-est}), first we have 
\[
\chi -p =\left\{\begin{array}{cl}
    \dfrac{\alpha_1-\alpha_2}{2\varepsilon} \,\big(\bxi\cdot\bx -c + \varepsilon\big), & \bx\in\Omega_{\varepsilon,1},  \\[4mm]
    \dfrac{\alpha_1-\alpha_2}{2\varepsilon} \,\big(\bxi\cdot\bx -c - \varepsilon\big), & \bx\in\Omega_{\varepsilon,2},  \\[4mm]
    0,  & \bx\in \Omega\setminus \Omega_\varepsilon.
\end{array}\right.
\]
By a rotation of the coordinates, $\bx=(s,\by)$, it is easy to see that the domain $\Omega_{\varepsilon,1}$ is bounded by the hyper-planes $s=c-\varepsilon$ and $s=c$ and the hyper-surfaces $\varphi_1(s)$ and $\varphi_2(s)$ on $\partial\Omega$. Hence, we have
\[ 
    \int_{\Omega_{\varepsilon,1}} \big(\bxi\cdot\bx -c + \varepsilon\big)^2\,d\bx 
=\int_{c-\varepsilon}^{c}\int^{\varphi_2(s)}_{\varphi_1(s)}\big(s -c + \varepsilon\big)^2\,d\by\, ds
\leq D^{d-1} \int_{c-\varepsilon}^{c}\big(s -c + \varepsilon\big)^2\,\, ds = \dfrac{D^{d-1}}{3}\,\varepsilon^3.
\] 
In a similar fashion, 
\[\int_{\Omega_{\varepsilon,2}} \big(\bxi\cdot\bx -c - \varepsilon\big)^2\,d\bx \leq \dfrac{D^{d-1}}{3}\varepsilon^3.\]
The above two inequalities imply
\[
\|\chi - p\|^2_{0,\Omega}
 \leq \left(\dfrac{\alpha_1-\alpha_2}{2\varepsilon}\right)^2 \dfrac{2D^{d-1}}{3}\varepsilon^3 = \dfrac{D^{d-1}(\alpha_1-\alpha_2)^2\varepsilon}{6}.
 \]
This proves the inequality in (\ref{chi-est}) and, hence, the lemma.
\end{proof}

Assume that $u$ is a piece-wise smooth function with respect to the partition $\{\Omega_1,\Omega_2\}$ such that the jump of $u$ on the interface $\cP=\Omega_1\cap\Omega_2$ is a constant $\alpha_2-\alpha_1$, i.e., 
\[
\jump{u}_{_\cP}\equiv u_2|_{_\cP} - u_1|_{_\cP}=\alpha_2-\alpha_1.
\]
Then $u$ has the following decomposition
\begin{equation}\label{decop}
u= \chi(\bx;\bxi,c) + \hat{u},
\end{equation}
where $\bxi$ is a vector normal to ${\bm \beta}$.
It is easy to see that $\hat{u}$ is continuous in $\Omega$ and piece-wise smooth.

\begin{theorem}
Assume that the advection velocity field ${\bm\beta}$ is a constant vector field and that $f\in C(\Omega)$. Let $u$ and $u_{_N}$ be the solutions of problems {\em (\ref{ls})} and {\em (\ref{L-NN})}, respectively. Then we have 
 \begin{equation}\label{tau_1-error}
 \vertiii{u-u_{_N}}_{\bm\beta}
 \leq C\,\left(\big|\alpha_1-\alpha_2\big|\, \sqrt{\varepsilon} + \inf_{v\in \cM({\scriptsize\btheta},L)} \vertiii{\hat{u}-v}_{\bm\beta}
 \right),
 \end{equation}
where $\hat{u}\in C(\Omega)$ is given in {\em (\ref{decop})}.
\end{theorem}

\begin{proof}
The assumptions on ${\bm \beta}$ and $f$ imply that the exact solution $u$ has the decomposition in (\ref{decop}). Now, 
(\ref{tau_1-error}) is a direct consequence of Lemmas~3.1 and 3.2.
\end{proof}

Similar to \cite{cai2020deep}, we evaluate the LS functional numerically. To this end, let 
\[
{\cal T}=\{K :\, K\mbox{ is an open subdomain of } \Omega\}
\] 
be a partition of the domain $\Omega$. Then
\[
{\cal E}_{-}=\{E
=\partial K \cap \Gamma_-:\,\, K\in\mathcal{T}\}
\]
is a partition of the inflow boundary $\Gamma_-$.
Let $\bx_{_K}$ and $\bx_{_E}$ be the centroids of $K\in {\cal T}$ and $E\in {\cal E}_-$, respectively. Define the discrete LS functional as follows:
\begin{equation}\label{L-NN-d}
\begin{split}
     \mathcal{L}_{_{\small {\cal T}}}\big(v(\bx; {\small\btheta});{\bf f}\big) 
     = \sum_{K \in {\cal T}} \big(v_{\bm\beta} +\hat{\gamma}\, v-f \big)^2(\bx_{_K}; {\small\btheta})\,|K|+  \sum_{E\in {\cal E}_-} \big(|\bm{\beta} \cdot \bm{n}|(v-g)^2\big)(\bx_{_E}; {\small\btheta})|E|,
\end{split}
\end{equation}
where $|K|$ and $|E|$ are the $d$ and $d-1$ dimensional measures of $K$ and $E$, respectively.
Then the discrete least-squares approximation is to find ${u}^N_{_{\small {\cal T}}}(\bx,{\small\btheta}^*)\in \cM({\small\btheta},L)$ such that
 \begin{equation}\label{discrete_minimization_functional}
  \mathcal{L}_{_{\small {\cal T}}} \big({u}^N_{_{\small {\cal T}}}(\bx,{\small\btheta}^*);{\bf f}\big) 
  = \min\limits_{v\in \cM({\scriptsize\btheta},L)} \mathcal{L}_{_{\small {\cal T}}}\big(v(\bx;{\small\btheta});\,{\bf f}\big)
 = \min_{{\scriptsize \btheta}\in\R^{N}}\mathcal{L}_{_{\small {\cal T}}} \big(v(\bx; {\small\btheta});{\bf f}\big).
\end{equation}

\begin{lemma}
Let $u$, $u_{_N}$, and $u_{_\cT}^N$ be the solutions of problems {\em (\ref{ls})}, {\em (\ref{L-NN})}, and {\em (\ref{discrete_minimization_functional})}, respectively. Then there exists a positive constant $C$ such that
\begin{equation}\label{Cea-L-d}
   \vertiii{u-u^{_N}_{_\cT}}_{\bm\beta}
    \le C\,\left( \inf_{v\in \cM({\scriptsize\btheta},L)} \vertiii{u-v}_{\bm\beta}
    + \big|(\cL-\cL_{_\cT})(u_{_N}-u_{_\cT}^{\small N}; {\bf 0})\big|
    + \big|(\cL-\cL_{_\cT})(u-u_{_N}; {\bf 0})\big|
    \right).
\end{equation}
\end{lemma}

\begin{proof}
By the triangle inequality, the fact that $\cL_{_\cT}(u_{_\cT}^N; {\bf f}) \leq \cL_{_\cT}(u_{_N}; {\bf f})$, and the continuity of the homogeneous functional $\mathcal{L}\big(v;\,{\bf 0}\big)$ in (\ref{equiv}), we have
\begin{eqnarray*}
\dfrac12\,\cL_{_\cT}(u_{_N}-u_{_\cT}^N; {\bf 0})
&\leq & \cL_{_\cT}(u_{_N}-u; {\bf 0}) + \cL_{_\cT}(u-u_{_\cT}^N; {\bf 0})
= \cL_{_\cT}(u_{_N}; {\bf f}) + \cL_{_\cT}(u_{_\cT}^N; {\bf f})\\[2mm]
&\leq & 2\, \cL_{_\cT}(u_{_N}; {\bf f})
= 2\,\big((\cL_{_\cT}-\cL)(u_{_N}-u; {\bf 0})
+\cL(u_{_N}-u; {\bf 0})\big)
\\[2mm]
&\leq & 2\,(\cL_{_\cT}-\cL)(u_{_N}-u; {\bf 0})
+2M\,\vertiii{u-u_{_N}}^2_{\bm\beta} ,
\end{eqnarray*}
which, together with the coercivity of the homogeneous functional $\mathcal{L}\big(v;\,{\bf 0}\big)$ in (\ref{equiv}), implies that
\begin{eqnarray*}
\alpha\, \vertiii{u_{_N}-u_{_\cT}^{N}}^2_{\bm\beta}
&\le &  \mathcal{L}\big(u_{_N}-u_{_\cT}^{N};\,{\bf 0}\big)
= \big(\mathcal{L}-\cL_{_\cT}\big)\big(u_{_N}-u_{_\cT}^N;\,{\bf 0}\big) +\mathcal{L}_{_\cT}\big(u_{_N}-u_{_\cT}^N;\,{\bf 0}\big)\\[2mm]
&\le & \big(\mathcal{L}-\cL_{_\cT}\big)\big(u_{_N}-u_{_\cT}^N;\,{\bf 0}\big)
+ 4\,(\cL_{_\cT}-\cL)(u_{_N}-u; {\bf 0})
+4M\,\vertiii{u-u_{_N}}^2_{\bm\beta}.
\end{eqnarray*}
Now, (\ref{Cea-L-d}) is a direct consequence of the triangle inequality, the above inequality, and Lemma~3.1. This completes the proof of the lemma.
\end{proof}

Lemma 3.4 indicates that the total error of the LSNN approximation with numerical integration is bounded by the approximation error of the neural network and the error of the numerical integration. 

\section{Initialization of two-layer neural network}  

The nonlinear optimization in (\ref{L-NN-d}) usually has many solutions, and the desired one is obtained only if we start from a close enough first approximation. In this section, we briefly describe the initialization process introduced in \cite{LiuCai} for two-layer neural network. 

To this end, a two-layer ReLU NN with $n_1$ neurons produces the following set of functions:
\begin{equation}\label{ReLU-n}
 {\cal M}({\small\btheta},2) = \left\{c_0+\sum_{i=1}^{n_1} c_i\sigma(\bomega_i\cdot \bx -b_i)\, :\,  
c_i,\, b_i\in \R,\,\, \bomega_i\in \cS^{d-1} \right\},
 \end{equation}
where $\cS^{d-1}$ is the unit sphere in $\R^d$.
Let
\[
\varphi_{0}(\bx)=1
\quad\mbox{and}\quad 
\varphi_{i}(\bx)=\sigma(\bomega_i\cdot \bx -b_i)
\quad\mbox{for } i=1,...,n_1.
 \]
For a given input weights and bias
\[
\bomega=(\bomega_1, ... , \bomega_{n_1})
 \quad\mbox{and}\quad 
 {\bf b}=(b_1, ..., b_{n_1}),
 \]
problem (\ref{vp}) may be approximated by finding $u_{n_1}=\sum\limits^{n_1}_{i=0}c_i\varphi_{i}(\bx)$ such that
\begin{equation}\label{c}
    a(u_{n_1},\varphi_{j}) =f(\varphi_{j})
    \quad\mbox{for } j=0,1,...,n_1.
\end{equation}
for $j=0,1,...,n_1$. Let 
\[
A(\bomega,\bb)=\left(a(\varphi_{j}, \varphi_{i}) \right)_{(n_1+1)\times (n_1+1)}
\quad\mbox{and}\quad
F(\bomega,\bb)=\left(f(\varphi_{j}) \right)_{(n_1+1)\times 1},
\]
then the coefficients, ${\bf c}=(c_0,c_1, ..., c_n)$, of $u_{n_1}$ is the solution of the system of linear algebraic equations
\begin{equation}\label{linear_system}
    A(\bomega,\bb)\, \bc = F(\bomega,\bb).
\end{equation}

\begin{lemma}
Assume that hyper-planes $\{\bomega_i\cdot \bx =b_i\}_{i=1}^{n_1}$ are distinct. Then
the coefficient matrix $A(\bomega,\bb)$ is symmetric, positive definite.
\end{lemma}

\begin{proof}
Obviously, $A(\bomega,\bb)$ is symmetric. Positive definiteness of $A(\bomega,\bb)$ follows from the lower bound in (\ref{equiv}) and the linear independence of $\{\varphi_{i}\}^{n_1}_{i=0}$ (see Lemma 2.1 of \cite{LiuCai}).
\end{proof}

As discussed in \cite{LiuCai}, the (breaking) hyper-planes
\[
\cP_i:\, \bomega_i\cdot \bx -b_i=0
\quad\mbox{for } i=1,...,n_1
\]
and the boundary of the domain $\Omega$ form a physical partition of the domain $\Omega$. It is then natural to initialize the input weights $\bomega$ and bias $\bb$ such that the corresponding hyper-planes $\{\cP_i\}^{n_1}_{i=1}$ form a uniform partition of the domain $\Omega$. The initial for the output weights and bias $\bc$ may be chosen to be the solution of problem (\ref{linear_system}).

\section{Numerical Experiments}

In this section, we present numerical results for test problems with constant, piece-wise constant, or variable advection velocity fields. The solutions of these test problems are discontinuous along an interface which is a line segment, a piece-wise line segment, or a curve.

In all experiments, 
the integration mesh $\cT$ is obtained by uniformly partitioning the domain $\Omega$ into identical squares with mesh size $h=10^{-2}$. The directional derivative in the least-squares functional is approximated by the backward finite difference quotient
\begin{equation}\label{finite_diff}
    v_{\bm\beta}(\bx_{_K}) \approx \frac{v(\bx_{_K})-v\big(\bx_{_K} - \rho\bar{\bm{\beta}}(\bx_{_K}\big))}{\rho}
\end{equation}
where $\rho\in\R$ is chosen to be smaller than the integration mesh size $h$, and $\bar{\bm{\beta}}$ is the unit vector in the ${\bm\beta}$ direction. The minimization problem in  (\ref{L-NN-d}) is solved numerically by the Adam version of gradient descent \cite{kingma2015}, and variable learning rate is used during the training.

Let $u$ be the exact solution of problem (\ref{pde}) and $\bar{u}^N_{_\cT}$ be the LSNN approximation. 
Tables \ref{test1_1_table}--\ref{test3_table} report the numerical errors in the relative $L^2$, $V_{\bm\beta}$, and graph norms. In these tables, a network structure is expressed by 2-$n$-1 for a two-layer network with $n$ neurons, by 2-$n_1$-$n_2$-1 for a three-layer network with $n_1$ and $n_2$ neurons in the respective first and second layers, and so on. 
Figures \ref{test1_1_figure}--\ref{test3_figure} depict the traces of the exact solution and the numerical approximation on a plane perpendicular to both the $x_1x_2$-plane and the discontinuous interface, which accurately illustrate the quality of the numerical approximation.

\subsection{Problems with a constant advection velocity fields}

In this section, we present numerical results for two test problems with constant advection velocity fields whose solutions are piece-wise constants (see, e.g., \cite{liu2020adaptive}). 
A two-layer neural network is employed and the network is initialized by the method described in section 4.

\subsubsection{Discontinuity along a vertical line segment}

The first test problem is the equation in (\ref{pde}) with  the domain $\Omega =(0,2) \times (0,1)$, the inflow boundary $\Gamma_- = \{(x,0):\, x\in (0,2)\} $, a constant advection velocity field $\bm{\beta} = (0,1)^T$, $\gamma =f=0$, and the inflow boundary data $g(x)=0$ for $x\in (0,\pi/3)$ and $g(x)=1$ for $x\in (\pi/3,2)$. Let $\Omega_1=\{(x,y)\in \Omega:\, 0<x<\pi/3\}$ and $\Omega_2=\{(x,y)\in \Omega:\,\pi/3<x<2\}$, it is then easy to see that the exact solution is a piece-wise constant given by
\[ 
u(x,y)=\left\{ \begin{array}{ll}
 0, & (x,y)\in \Omega_1, \\[2mm]
 1, & (x,y)\in \Omega_2.
 \end{array}\right.
 \] 
The discontinuous interface is the vertical line $x=\pi/3$.

This problem was used to test various adaptive least-squares finite element methods in \cite{liu2020adaptive}. In particular, the discontinuous interface $x=\pi/3$ was chosen so that if the initial mesh does not align with the interface, so is the mesh generated by either global or local mesh refinements. 

Numerical results in \cite{liu2020adaptive} (see Fig. \ref{test1_1_compare_figure})  
showed that the conforming least-squares finite element method (C-LSFEM) exhibits the Gibbs phenomena even with very fine mesh and that the newly developed flux-based LSFEM in \cite{liu2020adaptive} using a pair of the lowest-order elements is able to avoid overshooting on an adaptively refined mesh.

\begin{figure}[htbp]\label{test1_1_compare_figure}
\centering
\subfigure[C-LSFEM vertical cross section]{
\begin{minipage}[t]{0.33\linewidth}
\centering
\includegraphics[width=1.6in]{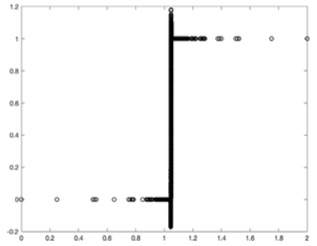}
\end{minipage}%
}%
\subfigure[flux-based LSFEM vertical cross section]{
\begin{minipage}[t]{0.33\linewidth}
\centering
\includegraphics[width=1.58in]{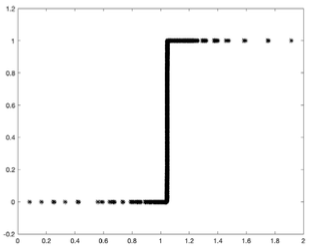}
\end{minipage}%
}%
\subfigure[An adaptively refined mesh]{
\begin{minipage}[t]{0.3\linewidth}
\centering
\includegraphics[width=1.6in]{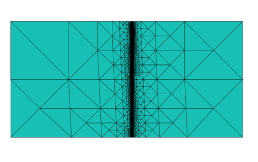}
\end{minipage}%
}%
\caption{Numerical results in \cite{liu2020adaptive} of the problem with discontinuity along a vertical line segment}
\end{figure}

The LSNN method is implemented with $\rho=h/2$ and a fixed learning rate $0.003$ with 20000 iterations. Our first set of experiments are done by using networks: 2-200-1 and 2-25-15-15-1. 
These two networks have $601$ and $705$ parameters, respectively, and provide good approximations (similar to Fig. \ref{test1_1_figure}(a,b)) to the exact solution.  

Lemma~3.2 indicates that a two-layer network with $2$ neurons is sufficient to approximate the exact solution well. Our second set of experiments are done by using networks: 2-2-1 and 2-4-1 with the respective $7$ and $13$ parameters. The 2-2-1 network fails to approximate the exact solution when the initial breaking lines are chosen to be the vertical line $x=1$ and the horizontal line $y=1/2$. This is because the iterative solver of the nonlinear optimization is not able to move the initial horizontal breaking line to the right place. The initial breaking lines for the 2-4-1 network are chosen to be the vertical lines $x=2/3$ and $x=4/3$ and the horizontal lines $y=1/3$ and $y=2/3$.

 \begin{table}[ht]
\caption{Relative errors of the problem with discontinuity along a vertical line segment}
\vspace{5pt}
\begin{tabular}{|l|l|l|l|l|}
\hline
Network structure  &$\frac{\|u-\bar{u}^N_{_\cT}\|_0}{\|u\|_0}$ &$\frac{\vertiii{u-\bar{u}^N_{_\cT}}_{\bm\beta}}{\vertiii{u}_{\bm\beta}}$ & $\frac{\mathcal{L}^{1/2}(\bar{u}^N_{_\cT};\bf f)}{\mathcal{L}^{1/2}(\bar{u}^N_{_\cT};\bf 0)}$ & Parameters \\ \hline
2-4-1  &0.058046&0.058304 & 0.050491   & 13\\ \hline
2-200-1 &0.058745& 0.058926 &0.048537   & 601\\ \hline
\end{tabular}
\centering
\label{test1_1_table}
\end{table}

\begin{figure}[htbp]
\centering
\subfigure[Numerical solution $\bar{u}^N_{_\cT}$]{
\begin{minipage}[t]{0.31\linewidth}
\includegraphics[width=1.8in]{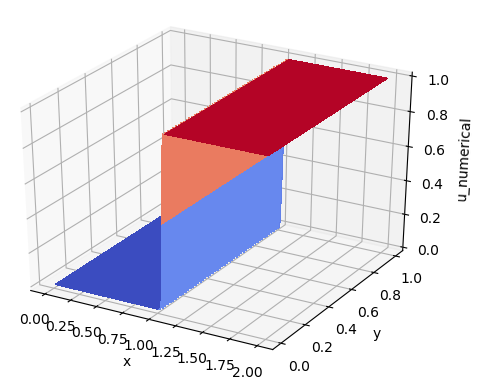}
\end{minipage}%
}%
\subfigure[Vertical cross section on $y=1$]{
\begin{minipage}[t]{0.32\linewidth}
\centering
\includegraphics[width=1.8in]{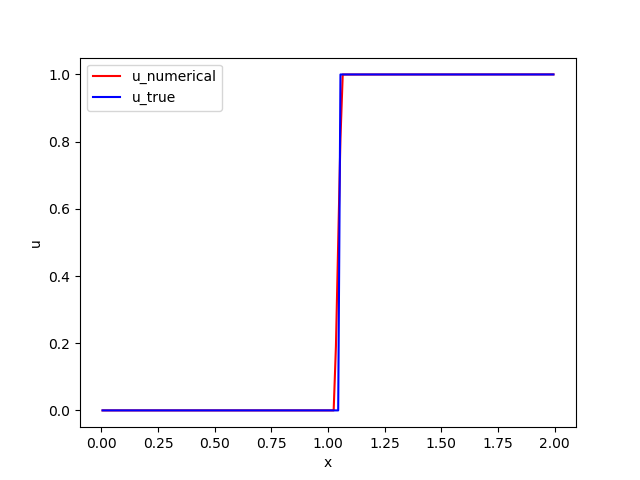}
\end{minipage}%
}%
\subfigure[Network breaking lines]{
\begin{minipage}[t]{0.32\linewidth}
\centering
\includegraphics[width=1.8in]{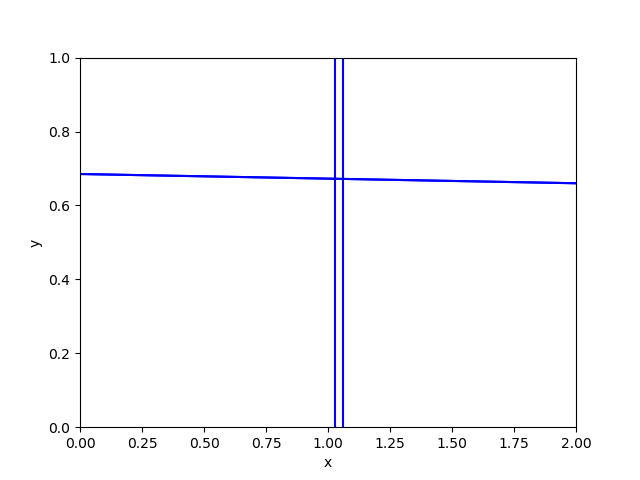}
\end{minipage}%
}%
\caption{Approximation results of the problem with discontinuity along a vertical line segment}\label{test1_1_figure}
\end{figure}

Errors of numerical results are presented in Table \ref{test1_1_table}. The second and third columns in Table \ref{test1_1_table} show that the approximation of the small network is slightly more accurate than that of the large network while the values of the loss functions are reversed. This indicates that the large network is trapped in a local minimum.
The numerical solution of the $4$-neuron network is depicted in Fig. \ref{test1_1_figure}(a). The traces of the exact and numerical solutions on the plane $y=1$ are depicted in Fig. \ref{test1_1_figure}(b), which shows no oscillation. Fig. \ref{test1_1_figure}(c) displays breaking lines of the network with two vertical lines $x=1.02882$ and $x=1.06114$ closing to the interface $x=\pi/3$. This indicates that breaking lines of neural network are capable of automatically adapting to the discontinuous interface. This simple test problem shows that the LSNN method out-performs the traditional mesh-based numerical methods.

\subsubsection{Discontinuity along the diagonal}

The second test problem is again equation (\ref{pde}) with a constant advection velocity vector and a piece-wise constant inflow boundary condition. Specifically, $\bm{\beta} = (1,1)^T/\sqrt2$, 
$\Omega =(-1,1)^2$, $\Gamma_-=\Gamma_-^1\cup \Gamma_-^2\equiv \{(-1,y):\, y \in (-1,1)\} \cup \{(x,-1):\, x \in (-1,1)\}$, $\gamma =1$, $g$ and $f$ are piece-wise constants given by
 \[
 g(x,y)=\left\{ \begin{array}{ll}
 1, & (x,y)\in \Gamma^1_-, \\[2mm]
 0, & (x,y)\in \Gamma^2_-,
 \end{array}\right.
 \quad\mbox{and}\quad 
 f(x,y) = \left\{ \begin{array}{ll}
 1, & (x,y)\in \Omega_1, \\[2mm]
 0, & (x,y)\in \Omega_2,
 \end{array}\right.
 \]
where $\Omega_1=\{(x,y)\in\Omega:\, y>x\}$ and $\Omega_2=\{(x,y)\in\Omega:\, y<x\}$. The exact solution of the test problem is $u(x,y) =f(x,y)$
with the discontinuous interface: $y=x$.
 
The LSNN method is implemented with $\rho=h/2$ and a fixed learning rate 0.003 with 20000 iterations for two networks: 2-4-1 and 2-6-1. The numerical results are presented in Table \ref{test1_2_table} which imply that the 2-4-1 network fails to accurately approximate the solution. Figure \ref{test1_2_figure} shows the NN approximation of the 2-6-1 network. The traces of the exact and numerical solutions on the plane $y=-x$ are depicted in Fig. \ref{test1_2_figure}(b). Clearly, the LSNN method with only 19 parameters approximates the exact solution accurately without the Gibbs phenomena. This test problem shows that the LSNN method is able to rotate and shift the initial breaking lines to approximate the discontinuous interface.
 
\begin{table}[htbp]
\centering
\caption{Relative errors of the problem with discontinuity along the diagonal}
\vspace{5pt}
\begin{tabular}{|l|l|l|l|l|}
\hline
Network structure  &$\frac{\|u-\bar{u}^N_{_\cT}\|_0}{\|u\|_0}$ &$\frac{\vertiii{u-\bar{u}^N_{_\cT}}_{\bm\beta}}{\vertiii{u}_{\bm\beta}}$ & $\frac{\mathcal{L}^{1/2}(\bar{u}^N_{_\cT};\bf f)}{\mathcal{L}^{1/2}(\bar{u}^N_{_\cT};\bf 0)}$ & Parameters \\ \hline
2-4-1 & 0.393864 & 0.393871 & 0.126095 & 13 \\ \hline
2-6-1 & 0.073534 &0.073826&0.067531   &19 \\ \hline
\end{tabular}
\centering
\label{test1_2_table}
\end{table}

\begin{figure}[htbp]
\centering
\subfigure[Network approximation $\bar{u}^N_{_\cT}$]{
\begin{minipage}[t]{0.31\linewidth}
\includegraphics[width=1.82in]{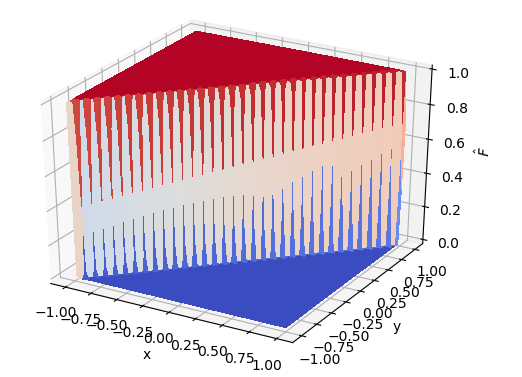}
\end{minipage}%
}%
\subfigure[Vertical cross section along $y=-x$
]{
\begin{minipage}[t]{0.3\linewidth}
\centering
\includegraphics[width=1.73in]{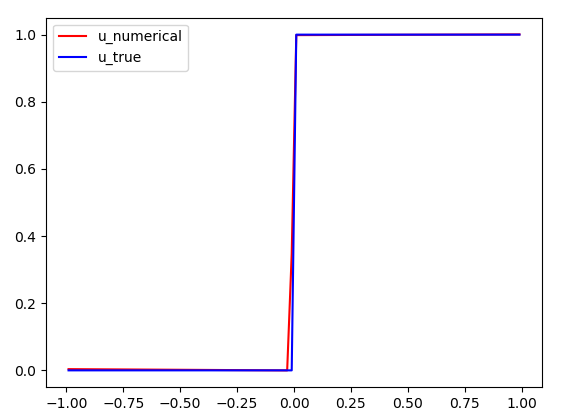}
\end{minipage}%
}%
\subfigure[Network breaking lines]{
\begin{minipage}[t]{0.3\linewidth}
\centering
\includegraphics[width=1.45in]{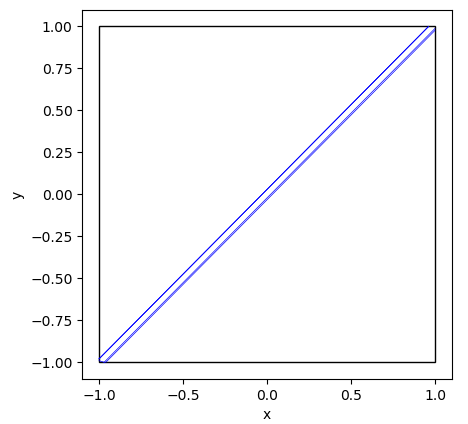}
\end{minipage}%
}%
\caption{Approximation results of the problem with discontinuity along the diagonal}\label{test1_2_figure}
\end{figure}

\subsection{Problem with a piecewise smooth solution}
The third test problem is a modification of the second test problem by changing the inflow boundary condition from the piece-wise constant to a discontinuous piece-wise smooth function and the domain from $\Omega=(-1,1)^2$ to $\Omega=(0,1)^2$, i.e., 
 \[
 g(x,y)=\left\{ \begin{array}{ll}
 \sin(y), & (x,y)\in \Gamma^1_-=\{(0,y):\, y \in (0,1)\} , \\[2mm]
 \cos(x), & (x,y)\in \Gamma^2_-=\{(x,0):\, x \in (0,1)\}.
 \end{array}\right.\]
Set $\gamma =f=0$, the exact solution of this test problem is
\[ 
u(x,y) = \left\{ \begin{array}{ll}
 \sin(y-x), & (x,y)\in \Omega_1=\{(x,y)\in (0,1)^2:\, y>x\}, \\[2mm]
 \cos(x-y), & (x,y)\in \Omega_2=\{(x,y)\in (0,1)^2:\, y<x\}.
 \end{array}\right.
 \] 
 
The LSNN method is employed with $\rho=h/2$ and a fixed learning rate 0.003 for 30000 iterations. Numerical results of three network models are reported in Table \ref{test_smooth_table} and the first two models fail to approximate the solution well. Figure \ref{test_smooth_figure} presents the NN approximation of the 2-40-1 network. The traces of the exact and numerical solutions on the plane $y=1-x$ are depicted in Fig. \ref{test_smooth_figure}(b), which exhibits no oscillation.
It is expected that the network with additional neurons is needed in order to approximate the solution well since the solution of the test problem is a piece-wise smooth function. Moreover, this test problem conforms Theorem~3.3 that a piece-wise smooth function having a constant jump along a line segment discontinuous interface may be approximated well by a two-layer network.

 \begin{table}[htbp]
\centering
\caption{Relative errors of the problem with a piece-wise smooth solution}
\vspace{5pt}
\begin{tabular}{|l|l|l|l|l|}
\hline
Network structure  &$\frac{\|u-\bar{u}^N_{_\cT}\|_0}{\|u\|_0}$ &$\frac{\vertiii{u-\bar{u}^N_{_\cT}}_{\bm\beta}}{\vertiii{u}_{\bm\beta}}$ & $\frac{\mathcal{L}^{1/2}(\bar{u}^N_{_\cT};\bf f)}{\mathcal{L}^{1/2}(\bar{u}^N_{_\cT};\bf 0)}$ & Parameters\\ \hline
2-20-1 & 0.110745  &0.110754 & 0.035571 & 61 \\ \hline
2-30-1 & 0.107525 & 0.107641 & 0.013568 &91 \\ \hline
2-40-1 &0.101411  & 0.101413 &0.003509 & 121 \\ \hline
\end{tabular}
\centering
\label{test_smooth_table}
\end{table}

\begin{figure}[htbp]
\centering
\subfigure[Network approximation $\bar{u}^N_{_\cT}$]{
\begin{minipage}[t]{0.31\linewidth}
\includegraphics[width=1.82in]{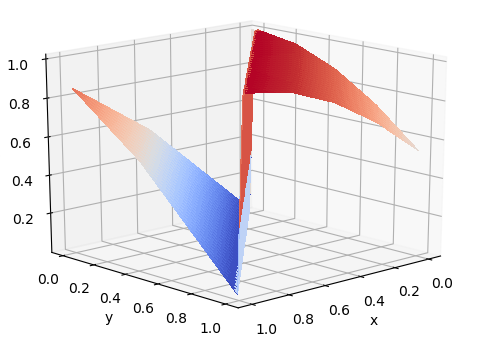}
\end{minipage}%
}%
\subfigure[Vertical cross section along $y=1-x$]{
\begin{minipage}[t]{0.299\linewidth}
\centering
\includegraphics[width=1.7in]{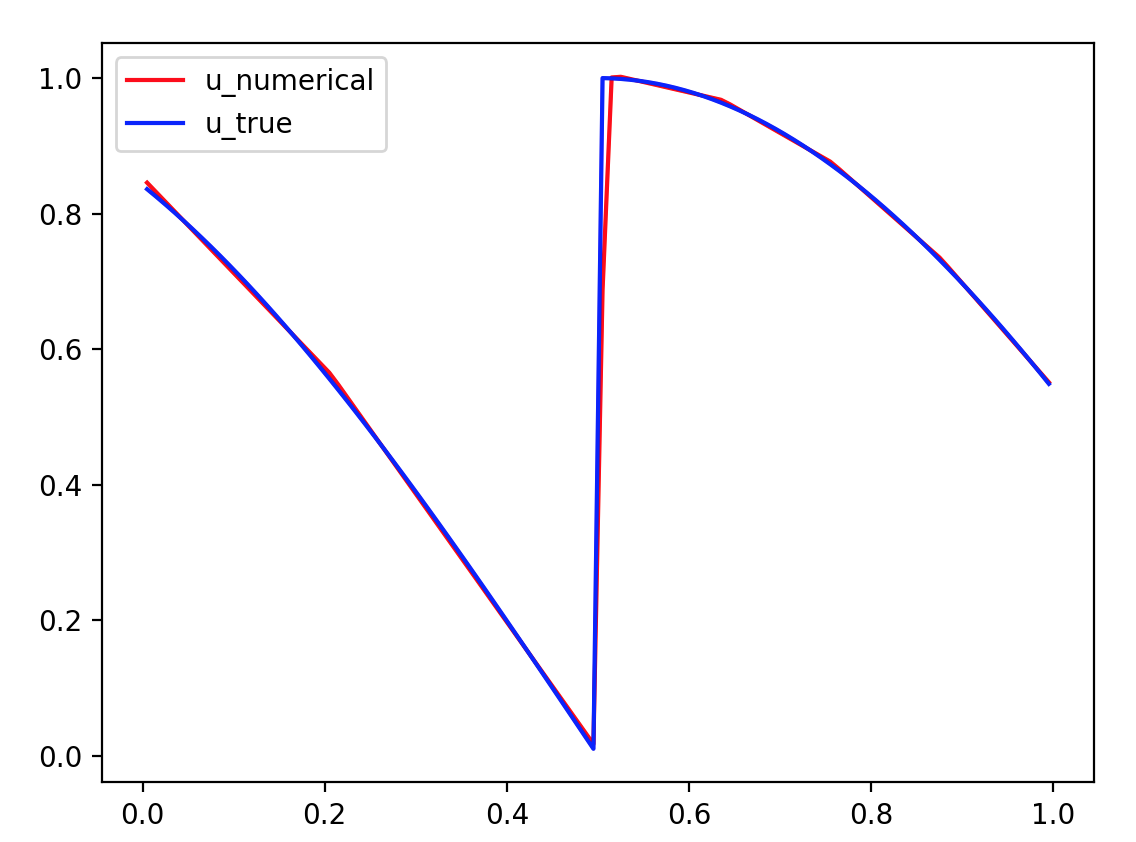}
\end{minipage}%
}%
\subfigure[Network breaking lines]{
\begin{minipage}[t]{0.3\linewidth}
\centering
\includegraphics[width=1.38in]{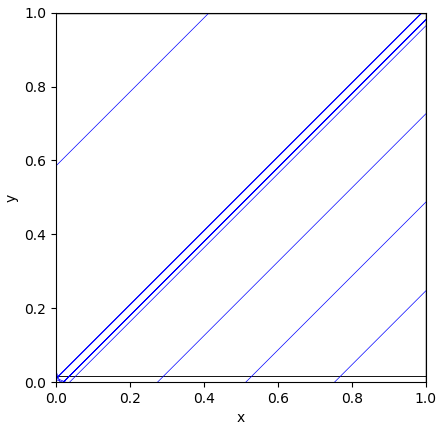}
\end{minipage}%
}%
\caption{Approximation results of the problem with a piece-wise smooth solution}\label{test_smooth_figure}
\end{figure}

\subsection{Problem with two discontinuous interfaces}
The fourth test problem is again a modification of the second test problem by changing the domain to $\Omega=(-1,1)\times(0,1)$, the inflow boundary condition to a combination of jumps and smooth function
\[
 g(x,y)=\left\{ \begin{array}{ll}
 \sin\left(\frac{\pi(x-y+0.9)}{0.3}\right), & (x,y)\in \Gamma^1_-=\{(x,0):\, x \in (-0.9,-0.6)\} , \\[2mm]
  -1, & (x,y)\in \Gamma^2_-=\{(x,0):\, x \in (-0.2,0.1)\} , \\[2mm]
 0, & (x,y)\in \Gamma_- \setminus (\Gamma^1_-\cup \Gamma^2_-)
 \end{array}\right.\]
 with the inflow boundary 
 \[\Gamma_- = \{(x,0):\, x \in (-1,1)\} \cup \{(-1,0)\} \cup \{(-1,y):\, y \in (0,1)\}.
 \]
Set $f$ as 
\[
 f(x,y)=\left\{ \begin{array}{ll}
 \sin\left(\frac{\pi(x-y+0.9)}{0.3}\right), & (x,y)\in \Upsilon_1=\{(x,y)\in \Omega:\, -0.9<x-y <-0.6\} , \\[2mm]
  -1, & (x,y)\in \Upsilon_2=\{(x,y)\in \Omega:\, -0.2<x-y <0.1\} , \\[2mm]
 0, & (x,y)\in \Omega \setminus (\Upsilon_1\cup \Upsilon_2),
 \end{array}\right.\]
  then the exact solution of the test problem is $u(x,y)=f(x,y)$ with two discontinuous interfaces $y=x+0.2$ and $y=x-0.1$, respectively.

   The LSNN method is implemented with $\rho= h/2$ and an adaptive learning rate which starts with 0.01 and decreases by 0.002 for every 20000 iterations. The total number of iterations is 80000. We observed from the experiment that adding a weight $\alpha$ to the inflow boundary loss in (\ref{L-NN-d}) is helpful for the training. Empirically, we choose $\alpha =10$ and report the numerical results for three respective network structures in Table \ref{test_smooth1_table}. The results suggest that the first 2-20-1 network model fails to approximate the solution well due to the possibility of training and/or insufficient number of neurons. Starting with a 2-30-1 network and applying the adaptive neuron enhancement strategy \cite{LiuCai} once, the 2-34-1 network provides an accurate approximation (see Table \ref{test_smooth1_table} and Figure \ref{test_smooth1_figure}). The traces of the exact and numerical solutions are depicted on the plane $y=0.8$ in Figure \ref{test_smooth1_figure}(c). 
   This test problem shows that the LSNN method using a small number of DoF is capable of approximating a discontinuous solution containing a smooth extrema without oscillations.

\begin{table}[htbp]
\centering
\caption{Relative errors of the problem with two discontinuous interfaces}
\vspace{5pt}
\begin{tabular}{|l|l|l|l|l|}
\hline
Network structure  &$\frac{\|u-\bar{u}^N_{_\cT}\|_0}{\|u\|_0}$ &$\frac{\vertiii{u-\bar{u}^N_{_\cT}}_{\bm\beta}}{\vertiii{u}_{\bm\beta}}$ & $\frac{\mathcal{L}^{1/2}(\bar{u}^N_{_\cT};\bf f)}{\mathcal{L}^{1/2}(\bar{u}^N_{_\cT};\bf 0)}$ & Parameters\\ \hline
2-20-1 & 0.363573 & 0.392153 & 0.393907 &61 \\ \hline
2-30-1 &0.147767  & 0.152132 & 0.132542 &91 \\ \hline
2-34-1 &0.117451  &0.120213  &0.112463  &103 \\ \hline
\end{tabular}
\centering
\label{test_smooth1_table}
\end{table}

\begin{figure}[htbp]
  \centering 
    \subfigure[Exact solution $u$]{ 
    \includegraphics[width=2.2in]{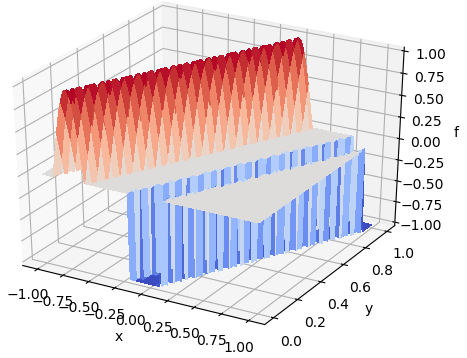}}
    \hspace{0.4in} 
    \subfigure[Network approximation $\bar{u}^N_{_\cT}$]{ 
    \includegraphics[width=2.2in]{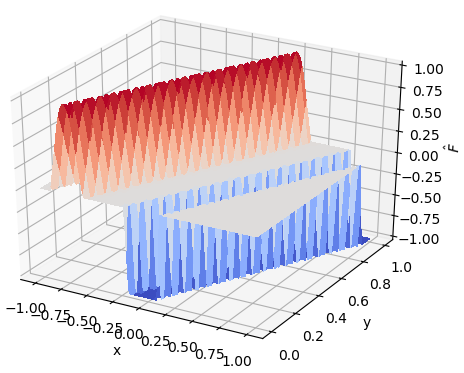}} 
\\
  \subfigure[Traces of the exact solution and approximation $\bar{u}^N_{_\cT}$ on the plane $y=0.8$]{ 
    \includegraphics[width=2.1in]{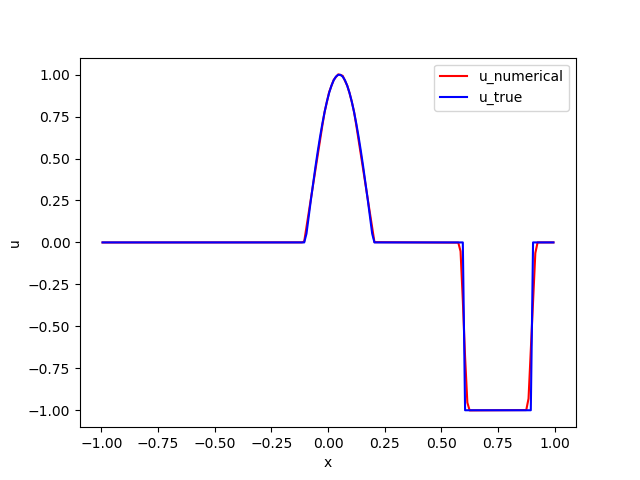}}
     \hspace{0.2in} 
\subfigure[Network breaking lines]{ 
    \includegraphics[width=2.4in]{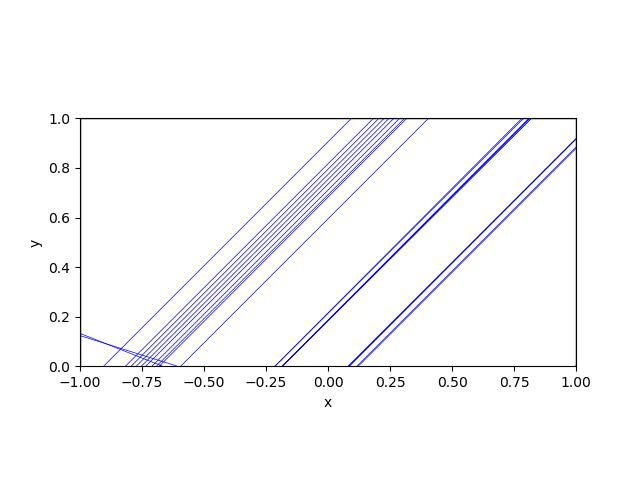}}
  \caption{Approximation results of the problem with two discontinuous interfaces}
  \label{test_smooth1_figure}
\end{figure}

\subsection{Problem with a piece-wise constant advection velocity field}

The fifth test problem is equation (\ref{pde}) defined on $\Omega =(0,1)^2$ with $\gamma =f =0$ and a piece-wise constant advection velocity field. Specifically, the advection velocity field is given by 
\begin{equation}\label{test_line_beta}
\bm{\beta} =\left\{ \begin{array}{rclll}
&(1-\sqrt2,1)^T, & (x,y)\in \Upsilon_1=\{(x,y)\in\Omega:\, y<x\}, \\[2mm]
&(-1,\sqrt2-1)^T, & (x,y)\in \Upsilon_2=\{(x,y)\in\Omega:\, y\ge x\}.
 \end{array}\right.
\end{equation}
and, hence, the inflow boundary of the problem is \begin{equation}\label{inflow-B}
\Gamma_-=\{(x,0):\, x\in (0,1)\} \cup \{(1,0)\}\cup \{(1,y):\, y\in (0,1)\}.
\end{equation}
Let $\Gamma^1_-=\{(x,0): x\in (0, a)\}$ with $a=43/64$. For the inflow boundary condition
\begin{equation}\label{test3_g}
g(x,y)=\left\{ \begin{array}{rl}
 -1,& (x,y)\in \Gamma^1_-, \\[2mm]
 1, & (x,y)\in \Gamma^2_-=\Gamma_-\setminus \Gamma_-^1,
 \end{array}\right.
 \end{equation}
the exact solution is a piece-wise constant: $u=-1$ in $\Omega_1$ and $u=1$ in $\Omega_2$, where $\Omega_2=\Omega\setminus\bar{\Omega}_1$ and 
 \[
 \Omega_1=\{\bx\in\Upsilon_1: \bxi_1 \cdot \bx < a\} \cup \{\bx\in\Upsilon_2: \bxi_2 \cdot \bx < a\}.
\]
Here, $\bxi_1=(1, \sqrt{2}-1)^T$ and $\bxi_2=(\sqrt{2}-1, 1)^T$ are vectors normal to 
$\bm{\beta}|_{_{\Upsilon_1}}$ and $\bm{\beta}|_{_{\Upsilon_2}}$, respectively. 

The LSNN method with $\rho = h/2$ and a fixed learning rate 0.003 with 50000 iterations is implemented for networks: 2-30-1, 2-200-1, and 2-5-5-1. Initialization of the first layer is done by the approach described in section 4, and that of the subsequent layers are randomly generated. The numerical results are presented in Table \ref{test2_table} and Figure \ref{test2_figure}, and the figures of the two-layer network is for the 2-200-1 model. The traces of the exact and numerical solutions on the plane $x=0$ and the breaking lines of these two networks are depicted in Fig. \ref{test2_figure}(c,d) and Fig. \ref{test2_figure}(e,f), respectively. 

Clearly, the two-layer network with 200 neurons (over 600 parameters) fails to approximate the solution well in average (see Table \ref{test2_table}) and point-wise (see Figure \ref{test2_figure}). A three-layer network with less than 8\% of parameters outperforms this large two-layer network in every aspects including breaking lines. Comparing these two networks, a three-layer network is more suitable than a two-layer network to accurately approximate the solution having a constant jump along a piece-wise line segment discontinuous interface.

\begin{remark}
Due to the random generation of some parameters, the training of 2-5-5-1 network is replicated five times and the best result is reported. We observe from the training process that the network may get trapped in a local minimum and fails to accurately approximate the solution. To address such issue, we introduce an adaptive process in \cite{cai2021multiadaptive} for obtaining a good initialization which is crucial for nonlinear optimization problems.
\end{remark}

 \begin{table}[htbp]
\centering
\caption{Relative errors of the problem with a piece-wise constant advection velocity field}
\vspace{5pt}
\begin{tabular}{|l|l|l|l|l|}
\hline
Network structure  &$\frac{\|u-\bar{u}^N_{_\cT}\|_0}{\|u\|_0}$ &$\frac{\vertiii{u-\bar{u}^N_{_\cT}}_{\bm\beta}}{\vertiii{u}_{\bm\beta}}$ & $\frac{\mathcal{L}^{1/2}(\bar{u}^N_{_\cT};\bf f)}{\mathcal{L}^{1/2}(\bar{u}^N_{_\cT};\bf 0)}$ & Parameters \\ \hline
2-30-1 &  0.487306 &0.556949 &  0.386919 & 91 \\ \hline
2-200-1 & 0.317839  & 0.402699&  0.259592 & 601 \\ \hline
2-5-5-1 & 0.086122  & 0.086131 &0.016945   & 46 \\ \hline
\end{tabular}
\centering
\label{test2_table}
\end{table}

\begin{figure}[htbp]
  \centering 
    \subfigure[2-layer NN approximation $\bar{u}^N_{_\cT}$]{ 
    \includegraphics[width=2.6in]{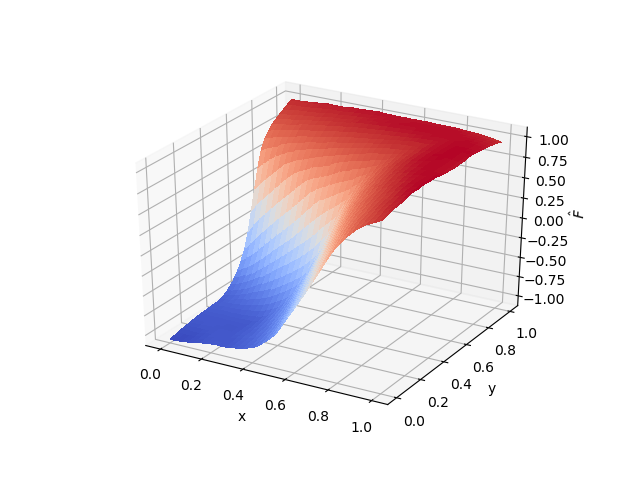}} 
  \subfigure[3-layer NN approximation $\bar{u}^N_{_\cT}$]{ 
    \includegraphics[width=2.6in]{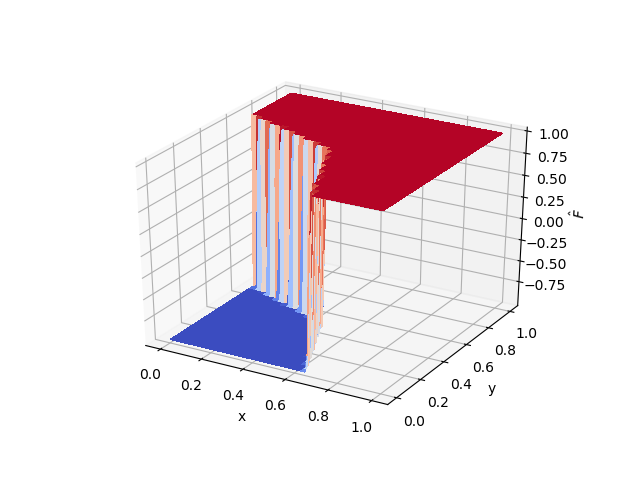}} 
\\
  \subfigure[2-layer NN vertical cross section on $x=0$]{ 
    \includegraphics[width=2.1in]{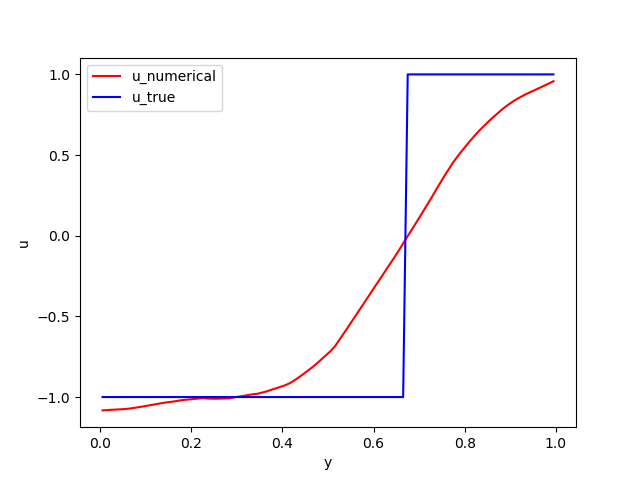}}
     \hspace{0.4in} 
\subfigure[3-layer NN vertical cross section on $x=0$]{ 
    \includegraphics[width=2.1in]{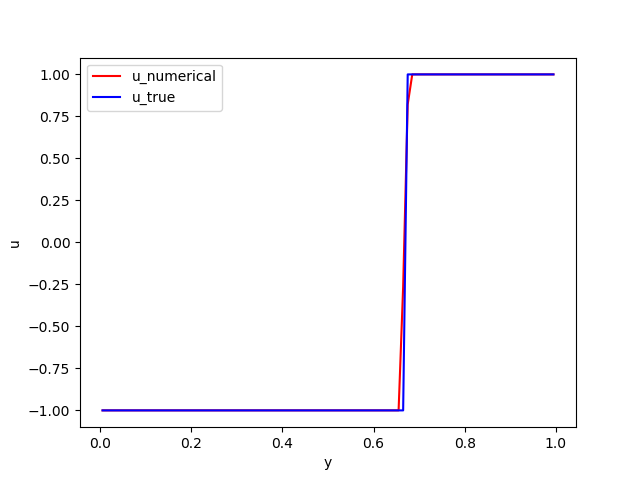}}
    \\
      \subfigure[2-layer NN breaking lines]{ 
    \includegraphics[width=1.7in]{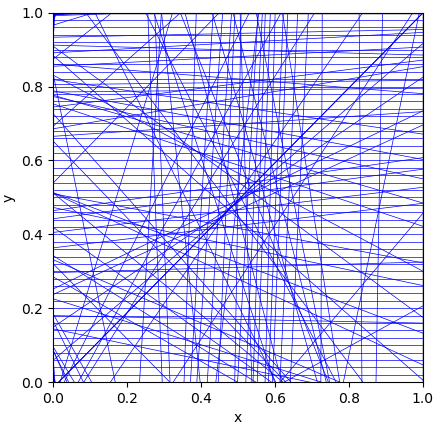}}
     \hspace{0.7in} 
\subfigure[3-layer NN breaking lines]{ 
    \includegraphics[width=1.7in]{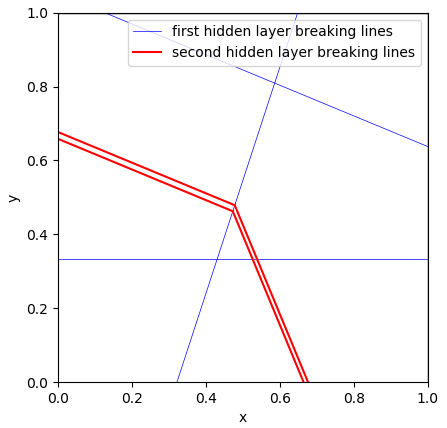}} 
  \caption{Approximation results of the problem with a piece-wise constant advection velocity field} 
  \label{test2_figure}
\end{figure}

Below we show theoretically that a three-layer neural network is sufficient for approximating the solution well (see Lemma~5.1 below). To make it slightly general, let  
\[
 \chi=\left\{\begin{array}{ll}
 \alpha_1, & \bx \in \Omega_1,\\[2mm]
 \alpha_2, & \bx \in \Omega_2.
 \end{array}
 \right.
 \]
Without loss of generality, assume that $\alpha_1 < \alpha_2$.
Let $p_1(\bx)$ and $p_2(\bx)$ be two-layer neural network functions given by
\[
p_i(\bx)=\alpha_1 +
\dfrac{\alpha_2-\alpha_1}{2\varepsilon} \Big(\sigma(\bxi_i\cdot\bx -a + \varepsilon) - \sigma(\bxi_i\cdot\bx -a-\varepsilon)\Big)
\]
for any $\varepsilon>0$ such that intersections between the domain $\Omega$ and the hyper-planes $\bxi_i\cdot\bx =a\pm\varepsilon$ are not empty.

\begin{lemma}
Let $p(\bx)=\max\{p_1(\bx),p_2(\bx)\}$, then we have 
 \begin{equation}\label{chi-est-2}
     \|\chi - p\|_{0,\Omega}
     =\left(\|\chi - p\|^2_{0,\Omega} + \|\chi_{\scriptsize{\bm\beta}} - p_{\scriptsize{\bm\beta}}\|^2_{0,\Omega}\right)^{1/2} 
     \leq \sqrt{\dfrac{2}{3}}\,D^{(d-1)/2} \big|\alpha_1-\alpha_2\big|\, \sqrt{\varepsilon},
 \end{equation}
where $D$ is the diameter of the domain $\Omega$.
\end{lemma}

\begin{proof}
Since $p(\bx)=p_i(\bx)$ in $\Upsilon_i$ for $i=1,\,2$ and $\Omega=\Upsilon_1\cup \Upsilon_2$, we have 
\[
\|\chi - p\|_{0,\Omega}^2=\|\chi - p_1\|_{0,\Upsilon_1}^2+\|\chi - p_2\|_{0,\Upsilon_2}^2.
\]
Combining with the fact that $\chi_{\scriptsize{\bm\beta}} - p_{\scriptsize{\bm\beta}}= 0$ in $\Omega$, 
(\ref{chi-est-2}) is then a direct consequence of Lemma~3.2.
\end{proof}

Similar as the discussion in \cite{he2018relu}, the maximum operation can be constructed by using an additional hidden layer of the ReLU network with 4 neurons:
\[
\max\{a,b\} = \dfrac{a+b}{2}+\dfrac{|a-b|}{2} = \bv \, \sigma\left(\bomega \left[\!\!\begin{array}{l}
a\\b\end{array}\!\!\right]\right)
\]
where the row vector and the $4\times 2$ matrix are given by
\[ \bv = \dfrac{1}{2}\,[1,-1,1,1] \quad \text{and} \quad \bomega=\left[\begin{array}{rr}
    1 & 1  \\
    -1 & -1 \\
    1 & -1 \\
    -1 & 1 \\
\end{array}\right],
\]
respectively.
Then this lemma indicates that a three-layer neural network is sufficient when the interface consists of two line segments.
\begin{remark}
In a similar fashion, a three-layer network can be constructed to approximate the solution with the interface consisting of more than two line segments.
\end{remark}

\subsection{Problem with a variable advection velocity field}

The last test problem is equation (\ref{pde}) defined on the domain $\Omega =(0,1)^2$ with a {\it variable} advection velocity field $\bm{\beta} = (-y,x)^T$ and $\gamma=f=0$
(see, e.g., \cite{bochev2016least,mu2018simple}).
With the inflow boundary condition $g$ given in (\ref{test3_g}),
the exact solution is a piece-wise constant given by
  \begin{equation}\label{test4_u}
u(x,y) = \left\{ \begin{array}{rl}
 -1,& (x,y)\in \Omega_1, \\[2mm]
 1,& (x,y)\in \Omega_2,
 \end{array}\right.
\end{equation}
where $\Omega_1=\{(x,y)\in\Omega:\, x^2+y^2<a^2\}$ and $\Omega_2=\{(x,y)\in\Omega:\, x^2+y^2 >a^2\}$.

For the LSNN method, again we use a uniform integration mesh $\cT$ with the mesh size $h=10^{-2}$; the finite difference quotient in (\ref{finite_diff}) is calculated with $\rho=h/10$
to avoid using values on both sides of the interface. Instead of intricately choosing the $\rho$ value, a robust approach will be developed in a forthcoming paper. Besides, the parameters are initialized by the method described in section 4 for the first layer and randomly for the subsequent layers. The learning rate starts with 0.005, and is reduced by half for every 50000 iterations. This learning rate decay strategy is used with 150000 iterations. Due to the random initialization of some parameters, 
numerical experiments are replicated three times and the best results for the three- and four-layer networks are reported in Table \ref{test3_table} and Figure \ref{test3_figure}. The traces of the exact and numerical solutions at the plane $x=0$ are depicted in Fig. \ref{test3_figure} (b) and (c) for the respective three- and four-layer networks. As shown in Fig. \ref{test3_figure} (b), the LSNN approximation of the three-layer network with 40 neurons at each layer smears the discontinuity. A careful examination of the iterative process, it seems to us that the smear is due to the initialization (see Fig. \ref{test4_figure1}).

 \begin{table}[htbp]
\centering
\caption{Relative errors of the problem with a variable advection velocity field}
\vspace{5pt}
\begin{tabular}{|l|l|l|l|l|}
\hline
Network structure  &$\frac{\|u-\bar{u}^N_{_\cT}\|_0}{\|u\|_0}$ &$\frac{\vertiii{u-\bar{u}^N_{_\cT}}_{\bm\beta}}{\vertiii{u}_{\bm\beta}}$ & $\frac{\mathcal{L}^{1/2}(\bar{u}^N_{_\cT};\bf f)}{\mathcal{L}^{1/2}(\bar{u}^N_{_\cT};\bf 0)}$ & Parameters \\ \hline
2-40-40-1 &  0.146226 &0.187823 & 0.108551 & 1761 \\ \hline
2-30-30-30-1 & 0.109266 & 0.122252& 0.039993 &1951 \\ \hline
\end{tabular}
\centering
\label{test3_table}
\end{table}

\begin{figure}[htbp]\label{test3_figure}
\centering
\subfigure[4-layer network approximation]{
\begin{minipage}[t]{0.29\linewidth}
\centering
\includegraphics[width=1.7in]{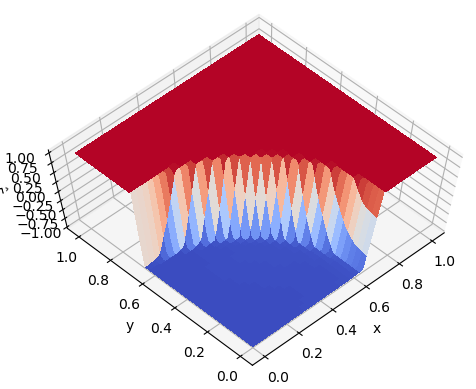}
\end{minipage}%
}%
\subfigure[3-layer NN cross section on $x=0$]{
\begin{minipage}[t]{0.34\linewidth}
\centering
\includegraphics[width=1.8in]{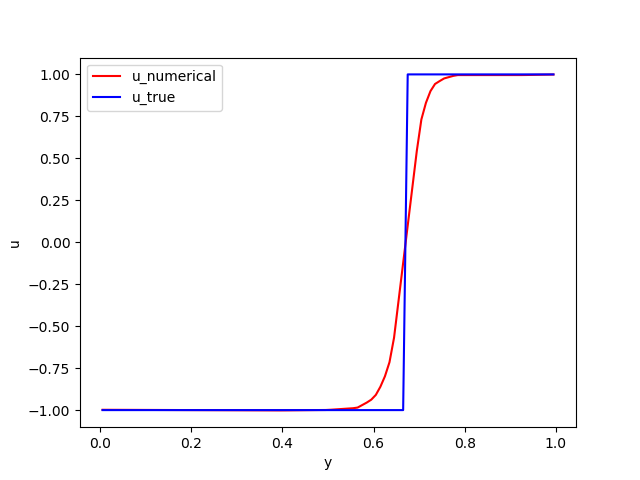}
\end{minipage}%
}%
\subfigure[4-layer NN cross section on $x=0$]{
\begin{minipage}[t]{0.34\linewidth}
\centering
\includegraphics[width=1.8in]{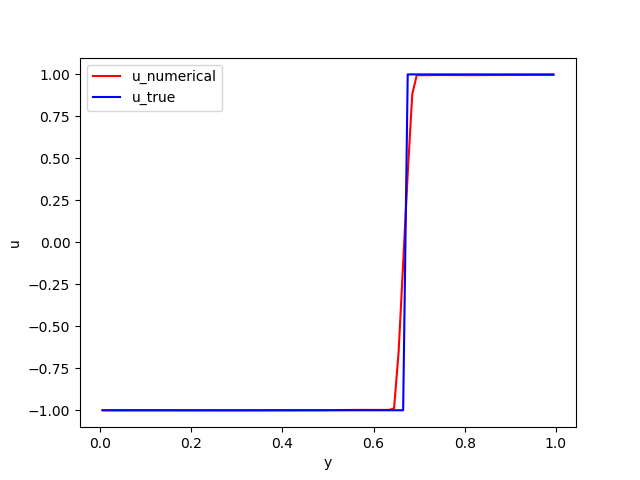}
\end{minipage}%
}%
\caption{Approximation results of the problem with a variable advection velocity field}
\end{figure}

\section{Method of model continuation}

As observed from our numerical experiments for the test problem with a curved discontinuous interface, initial of the parameters plays an important role in training neural networks. This is because the high dimensional nonlinear optimization usually have many solutions. Without a good initial, our previous simulations rely on over-parameterized neural networks to approximate the underlying problem well. The strategy of over-parameterization is computationally expensive.

Based on our numerical experiments in the previous sections, to generate a good initial for the parameters, we introduce the method of continuation through models for the advection-reaction problem in (\ref{pde}) with a variable advection velocity field ${\bm \beta}(\bx)$. 
To this end, let $\{{\bm\beta}_n(\bx)\}$ be a sequence of piece-wise constant vector fields. Consider the following advection-reaction problem with the advection velocity field ${\bm\beta}_n(\bx)$:
\begin{equation}\label{pde-n}
    \left\{\begin{array}{rccl}
    (u_n)_{{\bm\beta}_n} + \hat{\gamma}\, u_n  &= & f, &\text{ in }\,\, \Omega, \\
    u_n&=&g, &\text{ on }\,\, \Gamma_{-}.
    \end{array}\right.
\end{equation}
Let $u$ be the solution of (\ref{pde}), it is easy to see that $u-u_n$ satisfies
\begin{equation}\label{pde-err}
    \left\{\begin{array}{rccl}
    (u-u_n)_{{\bm\beta}_n} + \hat{\gamma}\, (u-u_n)  &= & u_{{\bm\beta}_n}-u_{{\bm\beta}}, &\text{ in }\,\, \Omega, \\
    u-u_n&=&0, &\text{ on }\,\, \Gamma_{-},
    \end{array}\right.
\end{equation}
which, together with the stability estimate in (\ref{stability}), implies 
\[
\|u-u_n\|_{0,\Omega}
\leq \vertiii{u-u_n}_{{\bm\beta}_n} \leq C\, \|u_{{\bm\beta}_n}-u_{{\bm\beta}}\|_{0,\Omega} =C\,\left(\int_\Omega \big(({\bm\beta}_n-{\bm\beta})\cdot\nabla u
\big)\,d\bx
\right)^{1/2}.
\]
Hence, if ${\bm\beta}_n$ is a good approximation to ${\bm\beta}$, then $u_n$ is a good approximation to $u$. This indicates that (\ref{pde-n}) provides a continuation process on the parameter $n$ for (\ref{pde}). 

For the test problem in section~5.4, since streamlines of the advection velocity field ${\bm\beta}=(-y,x)^T$ are quarter circles in $\Omega=(0,1)^2$ oriented counterclockwise, 
it is natural to approximate the quarter-circle by $n$ line segments. 
To this end, let $t_i=\dfrac{i\pi}{2n}$ for $i=0,1,...,n$ and \[
\Upsilon_{i+1}=\{(x,y)\in \Omega:\, (\sin t_i)x < (\cos t_i)y\,\,\mbox{ and }\,\, (\sin t_{i+1})x \ge (\cos t_{i+1})y\}.
\]
Then $\{\Upsilon_{i+1}\}_{i=0}^{n-1}$ forms a partition of $\Omega$ (see Fig. \ref{line_segment} for $n=4$).
This type of approximations leads to  
\[
\bm{\beta}_n = (\cos t_{i+1}-\cos t_{i}, \sin t_{i+1}-\sin t_{i})^T
\quad\mbox{in } 
\Upsilon_{i+1}
\]
for $i=0,1,...,n-1$. Note that ${\bm \beta}_2$ is the same vector field given in (\ref{test_line_beta}). Hence, (\ref{pde-n}) with $n=2$ and the test problem in section~5.3 are the same.

The method of model continuation starts with a three-layer neural network (2-5-5-1) to approximate $u_2$ (see the third row of Table \ref{test2_table} and Fig. \ref{test2_figure} (b,d)). This trained network is used as an initial for the parameters in the hidden layers of the 2-6-6-1 network to approximate $u_3$ by randomly generated the parameters of new neurons. The initial for the output weights and bias may be chosen as the solution of the system (\ref{linear_system}). The adaptive learning rate strategy which starts with 0.01 and decays by 20\% for every 50000 iterations is implemented with the method. The networks for $u_n$ with $n=4,5$ and for the test problem in section 5.4 are initialized sequentially in a similar fashion. Numerical results for approximating $u_n$ and $u$ are reported in Table \ref{test4_table}, and the traces of the exact and numerical solutions at the plane $x=0$ are depicted in Fig. \ref{test4_figure1}. 
The third and fourth columns show that the difficulty of the corresponding problems increase as the number of line segments increase. The fifth column shows that $u_n$ approaches to $u$ monotonously.
Comparing Table 5 with the last row of Table 6, it is clear that the method of model continuation is capable of reducing the size of the network significantly.

\begin{figure}[htbp]
    \centering
    \includegraphics[width=2.2in]{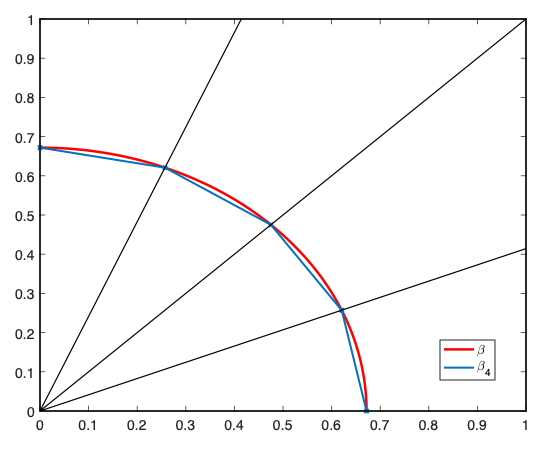}
    \caption{Discontinuous interface}
    \label{line_segment}
\end{figure}

\begin{table}[htbp]
\caption{Relative errors of the problem with discontinuity along line segments}
\vspace{5pt}
\begin{tabular}{|l|l|l|l|l|l|l|}
\hline
$n$&Network structure & $\frac{\|u_n-\bar{u}^N_{_\cT}\|_0}{\|u_n\|_0}$ &
$\frac{\vertiii{u_n-\bar{u}^N_{_\cT}}}{\vertiii{u_n}}$ &
$\frac{\|u-\bar{u}^N_{_\cT}\|}{\|u\|}$ &
$\frac{\mathcal{L}^{1/2}(\bar{u}^N_{_\cT};\bf f)}{\mathcal{L}^{1/2}(\bar{u}^N_{_\cT};\bf 0)}$ & Parameters \\ \hline
3&2-6-6-1& 0.075817 & 0.080026 &0.244483  &0.059422 &61 \\ \hline
4&2-6-6-1&0.104372  &0.110954 &0.216481  &0.064744 &61 \\ \hline
5&2-8-8-1& 0.097836  &0.109648 &  0.135606 &0.049938 &97 \\ \hline
curve&2-25-25-1&0.141261 &0.187616 & 0.141261&0.077233 & 726 \\ \hline
\end{tabular}
\centering
\label{test4_table}
\end{table}

\begin{figure}[htbp]\label{test4_figure1}
\centering
\subfigure[Vertical cross section of the \newline problem (\ref{pde-n}) with $n=3$ on the \newline plane $x=0$]{
\begin{minipage}[t]{0.3\linewidth}
\centering
\includegraphics[width=1.75in]{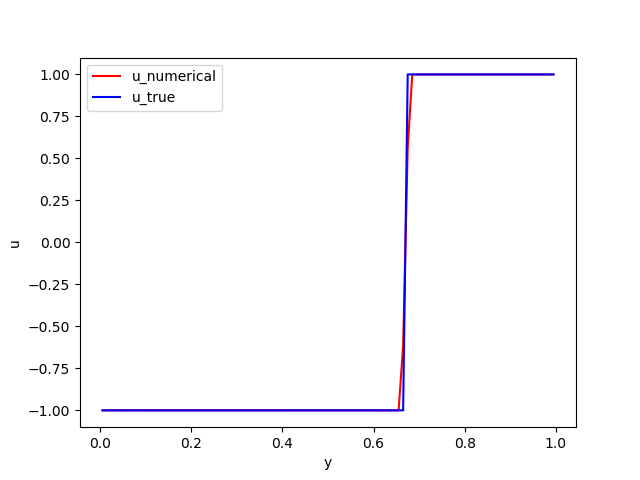}
\end{minipage}%
}%
\subfigure[Vertical cross section of the \newline problem (\ref{pde-n}) with $n=4$ on the \newline plane $x=0$]{
\begin{minipage}[t]{0.3\linewidth}
\centering
\includegraphics[width=1.75in]{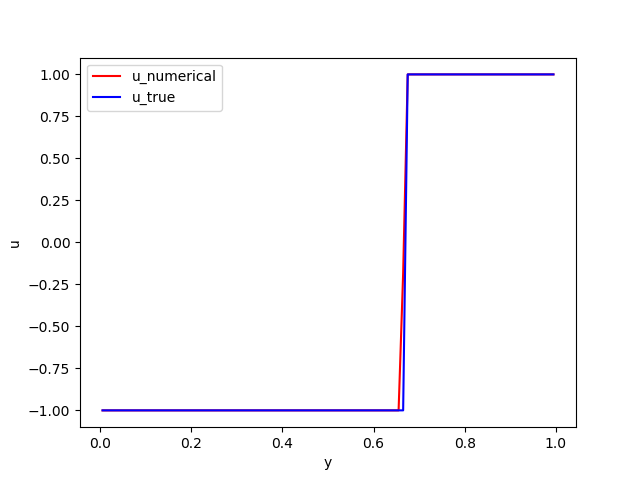}
\end{minipage}%
}%
\centering
\subfigure[Vertical cross section of the \newline problem (\ref{pde-n}) with $n=5$ on the \newline plane $x=0$]{
\begin{minipage}[t]{0.3\linewidth}
\centering
\includegraphics[width=1.75in]{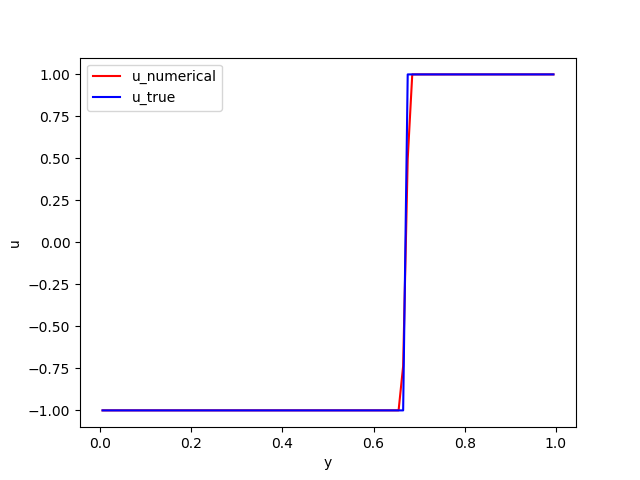}
\end{minipage}%
}%
\\
\subfigure[Vertical cross section of the original \newline problem (\ref{test4_u}) on the plane $x=0$]{
\begin{minipage}[t]{0.4\linewidth}
\centering
\includegraphics[width=1.8in]{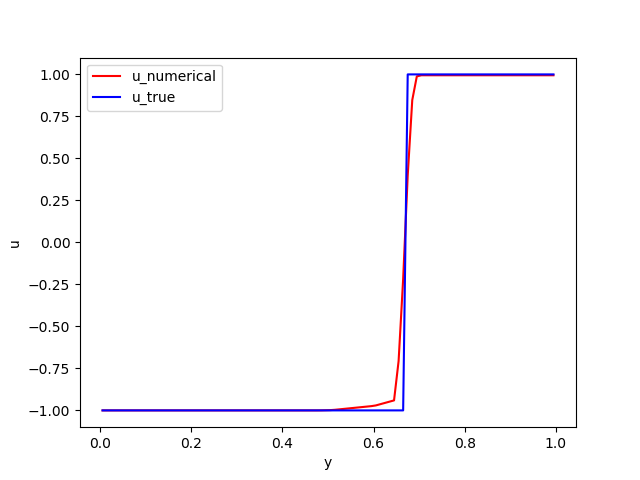}
\end{minipage}%
}%
\subfigure[Breaking lines of the original problem (\ref{test4_u})]{
\begin{minipage}[t]{0.4\linewidth}
\centering
\includegraphics[width=2.1in]{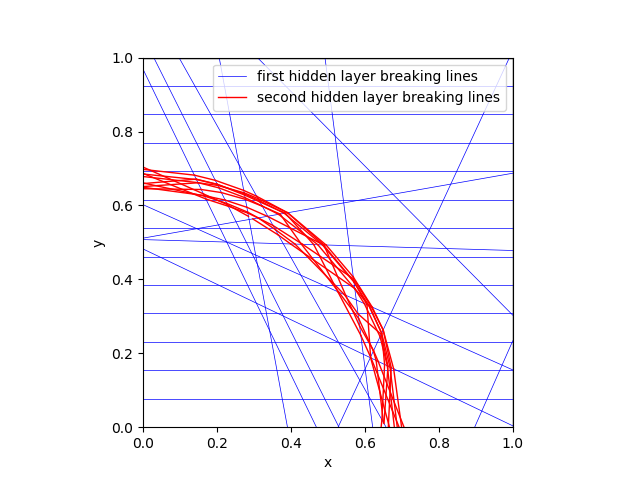}
\end{minipage}%
}%
\caption{Approximation results using the method of model continuation}
\end{figure}

\section{Discussions and Conclusions}

We proposed the LSNN method for solving the linear advection-reaction problem. The least-squares formulation, based on a direct application of the least-squares principle to the underlying problem, does not require additional smoothness of the solution if $f\in L^2(\Omega)$. In the $V_{\bm\beta}$ norm, the LSNN approximation is proved to be quasi-optimal, i.e., the error of the LSNN approximation is bounded above by the approximation error of the network.

A major challenge in numerical simulation of hyperbolic partial differential equations is the discontinuity of their solutions. For the linear transport problem in two dimensions, by decomposing the discontinuous solution into the sum of a piece-wise constant function and a continuous piece-wise smooth function, we are able to show theoretically and numerically that the LSNN method using a (at most) three-layer ReLU neural network is capable of approximating the discontinuous solution accurately without oscillation. In particular, the piece-wise constant solution can be approximated well by a ReLU network with a small number of neurons.

Numerical results presented in section 5 show that it is important to use a proper neural network in order to accurately approximate the solution of the underlying problem with fewer parameters. How to automatically design such a proper network, in terms of their width and depth, is an open and fundamental question for numerically solving partial differential equations within the prescribed accuracy. Following our recent paper on adaptive neuron enhancement method \cite{LiuCai}, this will be addressed in the forthcoming paper.

The procedure of training the value of the parameters is a problem in non-convex optimization which usually has many solutions and are complicated and computationally demanding. In order to obtain a desired solution, we introduced a method of model continuation for providing a good first approximation. Numerical results show that this method is effective for reducing the number of the parameters of the network. 
Moreover, a good initial is very helpful in training as well. 

Nevertheless, training is still a challenging problem since the learning rate of the methods of the gradient type is difficult to tune. A reasonably good learning rate can only be discovered through the method of trial and error. 
Using NNs to solve PDEs is relatively new, developing fast solvers is an open and challenging problem and requires lots of efforts from numerical analysts. Because of its great potential and many difficulties at the same time, machine learning is a hot research topic in scientific computing.


\bigskip
\bibliographystyle{ieee}
\bibliography{main.bbl}

\begin{thebibliography}{10}\itemsep=-1pt

\bibitem{bochev2001improved}
P.~Bochev and J.~Choi.
\newblock Improved least-squares error estimates for scalar hyperbolic
  problems.
\newblock {\em Computational Methods in Applied Mathematics}, 1(2):115--124,
  2001.

\bibitem{bochev2016least}
P.~Bochev and M.~Gunzburger.
\newblock Least-squares methods for hyperbolic problems.
\newblock In {\em Handbook of Numerical Analysis}, volume~17, pages 289--317.
  Elsevier, 2016.

\bibitem{bochev2001comparative}
P.~B. Bochev and J.~Choi.
\newblock A comparative study of least-squares, supg and galerkin methods for
  convection problems.
\newblock {\em International Journal of Computational Fluid Dynamics},
  15(2):127--146, 2001.

\bibitem{brezzi2004discontinuous}
F.~Brezzi, L.~D. Marini, and E.~S{\"u}li.
\newblock Discontinuous galerkin methods for first-order hyperbolic problems.
\newblock {\em Mathematical models and methods in applied sciences},
  14(12):1893--1903, 2004.

\bibitem{burman2009posteriori}
E.~Burman.
\newblock A posteriori error estimation for interior penalty finite element
  approximations of the advection-reaction equation.
\newblock {\em SIAM journal on numerical analysis}, 47(5):3584--3607, 2009.

\bibitem{cai2021multiadaptive}
Z.~Cai, J.~Chen, and M.~Liu.
\newblock Adaptive deep {ReLU} neural network: Best {LS} approximation and
  application to {PDE}s.
\newblock {\em submitted}, 2021.

\bibitem{cai2020deep}
Z.~Cai, J.~Chen, M.~Liu, and X.~Liu.
\newblock Deep least-squares methods: An unsupervised learning-based numerical
  method for solving elliptic {PDE}s.
\newblock {\em Journal of Computational Physics}, 420:109707, 2020.

\bibitem{carey1988least}
G.~F. Carey and B.~Jianng.
\newblock Least-squares finite elements for first-order hyperbolic systems.
\newblock {\em International journal for numerical methods in engineering},
  26(1):81--93, 1988.

\bibitem{dahmen2012adaptive}
W.~Dahmen, C.~Huang, C.~Schwab, and G.~Welper.
\newblock Adaptive petrov--galerkin methods for first order transport
  equations.
\newblock {\em SIAM journal on numerical analysis}, 50(5):2420--2445, 2012.

\bibitem{de2004least}
H.~De~Sterck, T.~A. Manteuffel, S.~F. McCormick, and L.~Olson.
\newblock Least-squares finite element methods and algebraic multigrid solvers
  for linear hyperbolic pdes.
\newblock {\em SIAM Journal on Scientific Computing}, 26(1):31--54, 2004.

\bibitem{de2005numerical}
H.~De~Sterck, T.~A. Manteuffel, S.~F. McCormick, and L.~Olson.
\newblock Numerical conservation properties of {H}(div)-conforming
  least-squares finite element methods for the burgers equation.
\newblock {\em SIAM Journal on Scientific Computing}, 26(5):1573--1597, 2005.

\bibitem{demkowicz2010class}
L.~Demkowicz and J.~Gopalakrishnan.
\newblock A class of discontinuous petrov--galerkin methods. part {I}: The
  transport equation.
\newblock {\em Computer Methods in Applied Mechanics and Engineering},
  199(23-24):1558--1572, 2010.

\bibitem{gottlieb1997gibbs}
D.~Gottlieb and C.-W. Shu.
\newblock On the gibbs phenomenon and its resolution.
\newblock {\em SIAM review}, 39(4):644--668, 1997.

\bibitem{he2018relu}
J.~He, L.~Li, J.~Xu, and C.~Zheng.
\newblock Re{LU} deep neural networks and linear finite elements.
\newblock {\em arXiv preprint arXiv:1807.03973}, 2018.

\bibitem{hesthaven2017numerical}
J.~S. Hesthaven.
\newblock {\em Numerical methods for conservation laws: From analysis to
  algorithms}.
\newblock SIAM, 2017.

\bibitem{hesthaven2007nodal}
J.~S. Hesthaven and T.~Warburton.
\newblock {\em Nodal discontinuous Galerkin methods: algorithms, analysis, and
  applications}.
\newblock Springer Science \& Business Media, 2007.

\bibitem{houston1999posteriori}
P.~Houston, J.~A. Mackenzie, E.~S{\"u}li, and G.~Warnecke.
\newblock A posteriori error analysis for numerical approximations of
  friedrichs systems.
\newblock {\em Numerische Mathematik}, 82(3):433--470, 1999.

\bibitem{houston2000posteriori}
P.~Houston, R.~Rannacher, and E.~S{\"u}li.
\newblock A posteriori error analysis for stabilised finite element
  approximations of transport problems.
\newblock {\em Computer methods in applied mechanics and engineering},
  190(11-12):1483--1508, 2000.

\bibitem{kingma2015}
D.~P. Kingma and J.~Ba.
\newblock Adam: A method for stochastic optimization.
\newblock In {\em International Conference on Representation Learning, San
  Diego}, 2015.

\bibitem{leveque1992numerical}
R.~J. LeVeque and R.~J. Leveque.
\newblock {\em Numerical methods for conservation laws}, volume~3.
\newblock Springer, 1992.

\bibitem{LiuCai}
M.~Liu, Z.~Cai, and J.~Chen.
\newblock Adaptive two-layer {R}e{LU} neural network.
\newblock {\em submitted}, 2021.

\bibitem{liu2020adaptive}
Q.~Liu and S.~Zhang.
\newblock Adaptive least-squares finite element methods for linear transport
  equations based on an {H}(div) flux reformulation.
\newblock {\em Computer Methods in Applied Mechanics and Engineering},
  366:113041, 2020.

\bibitem{mu2018simple}
L.~Mu and X.~Ye.
\newblock A simple finite element method for linear hyperbolic problems.
\newblock {\em Journal of Computational and Applied Mathematics}, 330:330--339,
  2018.

\bibitem{pinkus1999approximation}
A.~Pinkus.
\newblock Approximation theory of the mlp model in neural networks.
\newblock {\em Acta numerica}, 8(1):143--195, 1999.

\bibitem{raissi2019physics}
M.~Raissi, P.~Perdikaris, and G.~E. Karniadakis.
\newblock Physics-informed neural networks: A deep learning framework for
  solving forward and inverse problems involving nonlinear partial differential
  equations.
\newblock {\em Journal of Computational Physics}, 378:686--707, 2019.

\bibitem{shen2019deep}
Z.~Shen, H.~Yang, and S.~Zhang.
\newblock Deep network approximation characterized by number of neurons.
\newblock {\em arXiv preprint arXiv:1906.05497}, 2019.

\bibitem{siegel2020approximation}
J.~W. Siegel and J.~Xu.
\newblock Approximation rates for neural networks with general activation
  functions.
\newblock {\em Neural Networks}, 2020.

\bibitem{Sirignano18}
J.~Sirignano and K.~Spiliopoulos.
\newblock {DGM}: A deep learning algorithm for solving partial differential
  equations.
\newblock {\em Journal of Computational Physics}, 375:1139--1364, 2018.

\end{thebibliography}

\end{document}